\pdfoutput=1
\documentclass{amsart}

\usepackage[backend=biber, style=alphabetic]{biblatex}
\usepackage[svgnames]{xcolor}

\usepackage{tikz}
\usetikzlibrary{arrows.meta} 
\usepackage{mathtools}
\usepackage{tikz-cd}
\usepackage{amsmath, amssymb, amsthm}
\usepackage{here}
\usepackage{dynkin-diagrams}
\usetikzlibrary{intersections, patterns, snakes}
\usepackage{enumitem}
\usepackage{caption}
\usepackage{chngcntr}
\counterwithin{table}{section}
\counterwithin{figure}{section}

\numberwithin{equation}{section}
\mathtoolsset{showonlyrefs=true}

\theoremstyle{plain}

\newtheorem{theorem}{Theorem}[section]
\newtheorem{conjecture}[theorem]{Conjecture}
\newtheorem{lemma}[theorem]{Lemma}
\newtheorem{proposition}[theorem]{Proposition}
\newtheorem{corollary}[theorem]{Corollary}
\newtheorem{proposition-definition}[theorem]{Proposition-Definition}

\theoremstyle{definition}
\newtheorem{definition}[theorem]{Definition}
\newtheorem{example}[theorem]{Example}
\newtheorem{remark}[theorem]{Remark}

\DeclareMathOperator{\Hom}{\mathrm{Hom}}

\DeclareMathOperator{\Cone}{\mathrm{Cone}}

\DeclareMathOperator{\id}{\mathrm{id}}

\DeclareMathOperator{\Pic}{\mathrm{Pic}}

\DeclareMathOperator{\MCG}{\mathrm{MCG}}

\DeclareMathOperator{\Aut}{\mathrm{Aut}}

\DeclareMathOperator{\Supp}{\mathrm{Supp}}
\DeclareMathOperator{\SL}{\mathrm{SL}}
\DeclareMathOperator{\Spec}{\mathrm{Spec}}
\DeclareMathOperator{\Perf}{\mathrm{Perf}}

\DeclareMathOperator{\Ext}{\mathrm{Ext}}

\DeclareMathOperator{\sm}{\mathrm{sm}}
\DeclareMathOperator{\QCoh}{\mathrm{QCoh}}

\DeclareMathOperator{\GL}{\mathrm{GL}}

\DeclareMathOperator{\Tot}{\mathrm{Tot}}
\DeclareMathOperator{\Bl}{\mathrm{Bl}}
\DeclareMathOperator{\Coh}{\mathrm{Coh}}
\DeclareMathOperator{\Vect}{\mathrm{Vect}}
\DeclareMathOperator{\op}{\mathrm{op}}
\DeclareMathOperator{\Stab}{\mathrm{Stab}}

\DeclareMathOperator{\Sym}{\mathrm{Sym}}

\newcommand{\cA}{{\mathcal{A}}}
\newcommand{\cB}{{\mathcal{B}}}
\newcommand{\cD}{{\mathcal{D}}}

\newcommand{\cF}{{\mathcal{F}}}

\newcommand{\cL}{{\mathcal{L}}}

\newcommand{\cO}{{\mathcal{O}}}
\newcommand{\cP}{{\mathcal{P}}}
\newcommand{\cQ}{{\mathcal{Q}}}

\newcommand{\cT}{{\mathcal{T}}}

\newcommand{\cZ}{{\mathcal{Z}}}

\newcommand{\bA}{{\mathbb{A}}}
\newcommand{\bC}{{\mathbb{C}}}

\newcommand{\bG}{{\mathbb{G}}}

\newcommand{\bP}{{\mathbb{P}}}
\newcommand{\bQ}{{\mathbb{Q}}}
\newcommand{\bR}{{\mathbb{R}}}

\newcommand{\bZ}{{\mathbb{Z}}}

\newcommand{\fX}{{\mathfrak{X}}}

\newcommand{\fiber}{X}
\newcommand{\totalSpace}{\fX}
\newcommand{\baseSpace}{\bA^1}
\newcommand{\component}{Z}
\newcommand{\cptComponent}{Z_0}
\newcommand{\projectionMap}{\pi}
\newcommand{\closedImmersion}{i}
\newcommand{\cptComponentImmersion}{\iota}
\newcommand{\sectionMap}{s}
\newcommand{\canonicalBundleProjection}{p}
\newcommand{\canonicalBundleBaseSpace}{\bP^2}

\newcommand{\fiberCohSupp}{\Coh_0}
\newcommand{\fiberD}{\cD}
\newcommand{\fiberDc}{\cD_c}
\newcommand{\fiberPerf}{\Perf(\fiber)}
\newcommand{\fiberDsupp}{\cD_0}

\newcommand{\totalDsupp}{D^b_{\cptComponent}(\totalSpace)}
\newcommand{\simpleTotalDsupp}{\widetilde{\cD_0}}

\newcommand{\dpFunction}{\delta^{DP}_{\infty}}

\newcommand{\geometricChamber}{U(\fiber)}

\newcommand{\stabDagger}{\Stab^\dagger(\fiberDsupp)}

\newcommand{\comparisonIso}{\theta}
\newcommand{\totalGeometricChamber}{U(\totalSpace)}
\newcommand{\nTotalGeometricChamber}{V(\totalSpace)}
\newcommand{\totalStabDagger}{\Stab^\dagger(\simpleTotalDsupp)}

\newcommand{\boundaryE}{\partial_E G}

\newcommand{\positiveTwistWall}{W^{+}_E}
\newcommand{\negativeTwistWall}{W^{-}_E}

\newcommand{\slopeOfE}{\mu_E}
\newcommand{\discriminantOfE}{\Delta_E}
\newcommand{\slopeOfF}{\mu_F}
\newcommand{\discriminantOfF}{\Delta_F}

\newcommand{\sphericalTwistE}{T_E}
\newcommand{\halfTwistE}{H_E}
\makeatletter
\newcommand*{\rom}[1]{\expandafter\@slowromancap\romannumeral #1@}
\makeatother

\usepackage[
    colorlinks=true,
    citecolor=MediumBlue,
    urlcolor=MediumBlue,
    linkcolor=MediumBlue,
    breaklinks=true,
]{hyperref}

\title[Autoequivalences and stability conditions on a degenerate K3 surface]{
    Autoequivalences and stability conditions on a degenerate K3 surface
}
\author[H.~Arai]{Hayato Arai}
\address{
Graduate School of Mathematical Sciences,
The University of Tokyo,
3-8-1 Komaba,
Meguro-ku,
Tokyo,
153-8914,
Japan.
}
\email{hayato@ms.u-tokyo.ac.jp}
\begin{document}
\begin{abstract}
    We study autoequivalences and stability conditions on the derived category of coherent sheaves on a singular surface $X$ which arises as an open subvariety of a type III Kulikov degeneration of K3 surfaces.
    The surface $X$ consists of four irreducible components, one of which is $\bP^2$, and the others are non-compact rational surfaces.
    Using a comparison with the total space of the degeneration, we show that the connected component $\Stab^\dagger(D^b_{\bP^2}(X))$ of the space of stability conditions on the supported derived category $D^b_{\bP^2}(X)$ containing geometric stability conditions is simply connected, and describe its wall-and-chamber structure via half-spherical twists.
    As consequences, we determine the subgroup of the autoequivalence group $\Aut(D^b(X))$ that preserves this component; it is isomorphic to $\mathbb{Z} \times \Gamma_1(3) \times \Aut(X)$, where $\Gamma_1(3) \subset \mathrm{SL}(2,\mathbb{Z})$ is the congruence subgroup of level~3.
\end{abstract}
\maketitle
\section{Introduction}
\subsection{Background and motivation}\label{subsec:background-and-motivation}
Bridgeland \cite{MR2373143} introduced the notion of stability conditions on triangulated categories as a mathematical formulation of Douglas's $\Pi$-stability \cite{MR1957548} in string theory, thereby opening the way to studying autoequivalences through them \cite{MR2376815, MR2388559, MR3592689}.
While they have been widely studied for derived categories of smooth projective varieties, in recent years there has been increasing attention to the case of singular varieties.
In \cite{MR2155085, MR2264663}, Burban and Kreu{\ss}ler gave a complete description of stability conditions and autoequivalence groups of irreducible genus-one curves (namely, type $I_1$ and $II$ Kodaira fibers).
Karube \cite{MR4830066} studied stability conditions on the other types of Kodaira fibers, whose structures turn out to resemble those of K3 surfaces \cite{MR2376815} more closely than to those of elliptic curves.
The autoequivalence groups of such curves have only been determined in the case of type~$I_n$ fibers \cite{opperJEMS}.
Irreducible singular curves are also treated in \cite{liu2025stabilityconditionsirreducibleprojective}.
More recently, \cite{MR4813221, chou2024stabilityconditionsingularsurface} have constructed stability conditions on certain classes of irreducible singular surfaces using Bogomolov--Gieseker type inequalities, and \cite{MR4679958, haraJEMS} studied the space of stability conditions and Fourier--Mukai type autoequivalences for 3-fold flopping contractions with only Gorenstein terminal singularities.
In higher-dimensional cases, applications to the study of autoequivalence groups are still largely open, while \cite{MR2793026} ensures that a Bondal--Orlov type result \cite{MR1818984} still holds for Gorenstein varieties with ample (anti-)canonical bundle.

Motivated by these works, it is natural to turn to degenerations of K3 surfaces.
However, constructing stability conditions on projective varieties is generally difficult, and the presence of multiple compact components makes the analysis of autoequivalences harder.
Instead, in this paper we study an ``open version'' of such a degeneration.
More precisely, we consider a reducible surface $\fiber$ consisting of four components, one of which is $\bP^2$ and the others are $\Bl_0\bA^2$ (the blow-up of $\bA^2$ at the origin), glued along lines.
This surface arises as an open subvariety of a type III degeneration of K3 surfaces in Kulikov's classification \cite{MR506295}.
For the space of stability conditions $\Stab(D^b_{\bP^2}(\fiber))$ on the supported derived category $D^b_{\bP^2}(\fiber)$, we give a complete description of the geometric chamber $\geometricChamber$ and construct generic elements of its walls.
We then show that the group of autoequivalences of the full derived category $D^b(X)$ acts naturally on $\Stab(D^b_{\bP^2}(\fiber))$, and determine the subgroup that preserves the connected component containing $\geometricChamber$.

Our approach adapts the framework of Bayer and Macr\`i \cite{bayer-macri}, where stability conditions and autoequivalences on the total space of a smoothing family $\totalSpace = \Tot(K_{\canonicalBundleBaseSpace})$ of $\fiber$ were extensively studied.
We compare the space of stability conditions on $\fiber$ with that on $\totalSpace$ using the technique developed in \cite{MR2448280}, and then analyze the wall-and-chamber structure with respect to geometric stability conditions and the autoequivalence group.
While many of the underlying ideas are inspired by \cite{bayer-macri}, several aspects behave differently in the smooth threefold case and in our singular surface case (for example, Serre duality and Euler pairing cannot be applied to the entire derived category).
The most significant difference is that the spherical twists used in \cite{bayer-macri} are no longer available in our setting.
Instead, we employ the notion of \emph{half-spherical twists} (see Section \ref{subsec:half-spherical-twist} for the definition), introduced by the author in \cite{arai2024halfspherical}, which allow us to describe the wall-and-chamber structure in the singular surface case.
A related construction was obtained simultaneously and independently in \cite{MR4830066}; it may be viewed as a special case of half-spherical twists.

\subsection{Results}\label{subsec:results}
We work over the field of complex numbers $\bC$.
Let $\totalSpace = \Tot(K_{\bP^2})$ be an open Calabi--Yau threefold with the projection $\canonicalBundleProjection \colon \totalSpace \to \canonicalBundleBaseSpace$, and denote the ample line bundle $\canonicalBundleProjection^*\cO_{\bP^2}(1)$ by $\cO_{\totalSpace}(1)$.
The total space gives a resolution of the quotient singularity $\bA^3/\mu_3 = \Spec H^0(\totalSpace, \cO_{\totalSpace})$ where $\mu_3 = \Spec \bC[\omega]/(1-\omega^3)$ acts by scalar multiplication.
The composite of the map $\bA^3/\mu_3 \to \bA^1, [(x_1,x_2,x_3)] \mapsto x_1x_2x_3$ with the resolution $\totalSpace \to \bA^3/\mu_3$ gives a degenerating family
\begin{equation}
    \projectionMap \colon \totalSpace \to \baseSpace
\end{equation}
of $(\bG_m)^2$ whose central fiber $\fiber$ is a singular surface consisting of one $\bP^2$ and three copies of $\Bl_0\bA^2$:
\begin{align}
    \fiber        & = \cptComponent \cup \component_1 \cup \component_2 \cup \component_3 \\
    \cptComponent & = \bP^2, \quad \component_i = \Bl_0\bA^2 \quad(i = 1, 2, 3),
\end{align}
where $\Bl_0\bA^2$ is the blow-up of $\bA^2$ at the origin.
Denote the inclusion $\cptComponent \hookrightarrow \fiber$ by $\cptComponentImmersion$ and
the ample line bundle $\cO_{\totalSpace}(1)\vert_{\fiber}$ by $\cO_{\fiber}(1)$.
Let $\fiberD = D^b(\fiber)$ be the bounded derived category of coherent sheaves on $\fiber$, and $\fiberDsupp = D^b_{\cptComponent}(\fiber) \subset \fiberD$ be the full subcategory of complexes supported on the compact component $\cptComponent$.

The first part of our results is a description of the space $\Stab(\fiberDsupp)$ of Bridgeland stability conditions on $\fiberDsupp$.
We compare it with the space $\Stab(D^b_{\cptComponent}(\totalSpace))$ studied in \cite{bayer-macri}, by constructing a comparison map $\theta \colon \Stab(D^b_{\cptComponent}(\totalSpace)) \to \Stab(\fiberDsupp)$ induced by the inclusion $\closedImmersion \colon \fiber \hookrightarrow \totalSpace$ (Theorem \ref{thm:comparison-isomorphism}).
Let $\geometricChamber \subset \Stab(\fiberDsupp)$ be the subset consisting of stability conditions on $\fiberDsupp$ that are \emph{geometric}
in the sense that skyscraper sheaves $\cO_x$ for all $x \in \cptComponent$ are stable of the same phase.
The forgetful map
\begin{equation}
    \cZ \colon \geometricChamber \to \Hom(K(\fiberDsupp), \bC) \cong \bC^3, \quad (Z, \cA) \mapsto Z
\end{equation}
has the image defined by explicit inequalities, and there is an explicit construction of the heart $\cA$ up to shift (Theorem \ref{thm:fiber-geometric-chamber}).
Next, we study the structure of the boundary $\partial \geometricChamber$,
which consists of real codimension one submanifolds of $\Stab(\fiberDsupp)$ called \emph{walls}.
We explicitly construct two types of walls $W_E^{\pm}$ associated with each exceptional bundle $E$ on $\cptComponent$, and show that there are no other walls on the boundary (Theorem \ref{thm:fiber-boundary}):
\begin{equation}
    \partial \geometricChamber = \bigcup_{E} (\positiveTwistWall \sqcup \negativeTwistWall).
\end{equation}
The other side of the wall $\positiveTwistWall$ (resp. $\negativeTwistWall$) is obtained by applying the half-spherical twist $H_E$ (see Section \ref{subsec:half-spherical-twist} for the definition) along the sheaf $\cptComponentImmersion_*E$,
and one has
\begin{equation}
    \positiveTwistWall = \overline{\geometricChamber} \cap H_E(\overline{\geometricChamber}), \quad \negativeTwistWall = \overline{\geometricChamber} \cap H_E^{-1}(\overline{\geometricChamber}).
\end{equation}
It follows that the connected component $\stabDagger$ of $\Stab(\fiberDsupp)$ containing $\geometricChamber$ is covered by the translates of $\overline{\geometricChamber}$ under compositions of half-spherical twists.
This is essential for studying the autoequivalence group $\Aut(\fiberD)$.
\begin{theorem}[Corollary \ref{cor:stab-dagger-is-covered-by-translates-of-geometric-chamber}]
    We have $\stabDagger = \bigcup_{\Phi} \Phi(\overline{\geometricChamber})$, where $\Phi$ runs over the subgroup of $\Aut(\fiberD)$ generated by the $H_{E}$ for all exceptional bundles $E$ on $\cptComponent$.
\end{theorem}
We also have the following topological property of $\stabDagger$:
\begin{theorem}[Corollary \ref{cor:fiber-stab-dagger-simply-connected}]
    The space $\stabDagger$ is simply connected.
\end{theorem}

Any autoequivalence of $\fiberD$ preserves the subcategory $\fiberDsupp$ (Proposition \ref{prop:auto-preserves-subcategories}), so that $\Aut(\fiberD)$ acts naturally on $\Stab(\fiberDsupp)$.
Let $\Aut^\dagger(\fiberD) \subset \Aut(\fiberD)$ be the subgroup consisting of autoequivalences preserving the connected component $\stabDagger$.
The main result of this paper is a description of $\Aut^\dagger(\fiberD)$:

\begin{theorem}[Proposition \ref{prop:automorphism-group-of-fiber} and Theorem \ref{thm:autoequivalence-group}]\label{thm:intro-autoequivalence-group}
    We have an isomorphism
    \begin{equation}
        \Aut^\dagger(\fiberD) \cong \bZ \times \Gamma_1(3) \times \Aut(\fiber),
    \end{equation}
    where
    \begin{enumerate}
        \item $\bZ$ is generated by the shift functor,
        \item $\Gamma_1(3) \subset \SL(2, \bZ)$ is the congruence subgroup of level 3, which is generated by $\alpha = H_{\cO_{\cptComponent}}$ and $\beta = - \otimes \cO_{\fiber}(1)$ with the relation $(\alpha\beta)^3 = 1$, and
        \item $\Aut(\fiber) \cong S_3 \ltimes (\bC^*)^3$ is the automorphism group of $\fiber$.
    \end{enumerate}
\end{theorem}
As conjectured by Bridgeland for K3 surfaces \cite[Conjecture 1.2]{MR2376815}, the component $\Stab^\dagger(\fiberDsupp)$ is expected to be preserved by the whole autoequivalence group $\Aut(\fiberD)$, and hence $\Aut^\dagger(\fiberD) = \Aut(\fiberD)$.
We make the following stronger conjecture:
\begin{conjecture}
    The space $\Stab(\fiberDsupp)$ is connected.
    In particular, we have $\Aut^\dagger(\fiberD) = \Aut(\fiberD)$ and Theorem \ref{thm:intro-autoequivalence-group} holds for $\Aut(\fiberD)$.
\end{conjecture}

\subsection{Relation to mirror symmetry}\label{subsec:relation-to-mirror-symmetry}
Another motivation for considering derived categories of singular varieties comes from homological mirror symmetry (HMS).
It predicts a derived equivalence between the Fukaya category $\cF(M)$ of a symplectic manifold $M$ and the category of coherent sheaves on a complex manifold $Y$ \cite{MR1403918}.
Although the original conjecture was stated for smooth Calabi--Yau varieties, Seidel's approach \cite{MR1957046, MR3364859, MR3294958} naturally involves degenerations of Calabi--Yau varieties as follows.
Let $(M, \omega)$ be a K\"ahler manifold and $D \subset M$ be a divisor Poincar\'e dual to $[\omega]$.
Seidel introduced the \emph{relative Fukaya category} $\cF(M, D)$, which gives a deformation of $\cF(M \setminus D)$ to $\cF(M)$.
This should be mirror to a deformation of a singular variety $X$ to a smooth variety $Y$.
The main idea of \cite{MR1957046} is to first prove an equivalence between $\cF(M \setminus D)$ and $D^b(X)$, and then extend it using deformation theory.

In that context, one expects a natural action of the graded symplectic mapping class group $\MCG(M,D)$ on $D^b(X)$, which would provide a powerful tool for studying $\Aut(D^b(X))$ (cf.~\cite{MR1831820}).
In the case of maximal degenerations of elliptic curves (i.e., type~$I_n$ Kodaira fibers), where the mirror equivalence has been established in \cite{lekili2012arithmeticmirrorsymmetry2torus, MR3663596}, such an action has been investigated in \cite{opperJEMS, arai2024halfspherical}.
In this situation, the action of \emph{half-twists} (in the sense of two-dimensional topology) is realized as half-spherical twists on the mirror side \cite[Theorem~1.3]{arai2024halfspherical}.
For K3 surfaces, maximal degenerations are given by type~III degenerations, and the HMS equivalence is studied in \cite{MR4718488, hacking2025homologicalmirrorsymmetryprojective}, where $M$ is a K3 surface and $D$ is a smooth ample divisor.
Our case $\fiber \subset \totalSpace$ can be regarded as an ``open version'' of such a degeneration (see also Remark~\ref{rem:compactification-to-type-III-degeneration}), in which $\cptComponent = \bP^2$ is the only compact component.
The half-spherical twists arising from exceptional bundles on $\cptComponent$ are expected to come from $\MCG(M, D)$, although we do not pursue this direction in the present paper.

\subsection{Future directions}\label{subsec:future-directions}
A natural next step is to treat the case where the central fiber $\fiber$ is itself a type~III degeneration of K3 surfaces (rather than an open subvariety).
In this situation, exceptional bundles on the irreducible components of $\fiber$ still induce half-spherical twists; we expect that, together with standard autoequivalences, these generate $\Aut D^b(\fiber)$.
Moreover, as explained in Section \ref{subsec:relation-to-mirror-symmetry}, they should correspond to certain symplectomorphisms of the pair $(M,D)$ via HMS, and be related to half-twists or their four-dimensional analogue yet to be identified.
We expect that such a correspondence will provide new insights into higher-dimensional symplectic mapping class groups (see also \cite{MR4850449} for this direction).

Our result may contribute to the study of \emph{stability conditions over a base} \cite{MR4292740} by providing an example of what happens on singular fibers.
For a flat family $\totalSpace \to S$, it is defined as a collection of stability conditions on each fiber $\totalSpace_s$, subject to suitable compatibility conditions.
In this framework, smooth fibers are relatively well understood, whereas the behavior on singular fibers remains largely mysterious.

\subsection{Organization of the paper}\label{subsec:organization-of-the-paper}
The paper is organized as follows.
In Section \ref{sec:preliminaries}, we review Bridgeland stability conditions and (half-)spherical twists.
Section \ref{sec:setting} fixes the geometric setup: we describe the family $\totalSpace$ and the central fiber $\fiber$.
In Section \ref{sec:various-derived-categories}, we introduce several subcategories of $\fiberD = D^b(\fiber)$ and study their relationships at the level of Grothendieck groups and autoequivalences.
Section \ref{sec:stability-conditions} is devoted to the study of the space of Bridgeland stability conditions on $\fiberDsupp = D^b_{\cptComponent}(\fiber)$.
In Section \ref{sec:autoequivalence-group}, we consider the structure of the autoequivalence group $\Aut(\cD)$.
Finally, Appendix \ref{sec:lemmas-on-integral-functors} contains some technical lemmas used in the main text.

\subsection{Acknowledgements}\label{subsec:acknowledgements}
The author would like to thank his advisor, Kazushi Ueda, for many helpful discussions and comments.
He is also grateful to Tomohiro Karube, Wahei Hara, and Yukinobu Toda for useful conversations.
This work is supported by JSPS KAKENHI Grant Number JP24KJ0684.
\subsection{Notations and conventions}\label{subsec:notations-and-conventions}
\begin{itemize}
    \item We always work over the field of complex numbers $\bC$. Every scheme is assumed to be defined over $\bC$ and every triangulated category is assumed to be $\bC$-linear. An autoequivalence means a $\bC$-linear exact autoequivalence.
    \item A \emph{variety} means a separated scheme of finite type over $\bC$, not necessarily irreducible or reduced.
    \item Unless otherwise specified, a \emph{point} of a variety means a closed point.
    \item The skyscraper sheaf at a point $x$ of a variety is denoted by $\cO_x$.
    \item For a variety $X$, we denote by $\Coh(X)$ the abelian category of coherent sheaves on $X$ and by $D^b(X)$ its bounded derived category. Similarly, we denote by $\QCoh(X)$ the abelian category of quasi-coherent sheaves on $X$ and by $D(\QCoh(X))$ its derived category. The full subcategory of $D(\QCoh(X))$ consisting of perfect complexes is denoted by $\Perf(X)$.
    \item We drop the symbols $L$ and $R$ from the notation of derived functors, e.g., we write $f_*$ instead of $Rf_*$.
\end{itemize}

\section{Preliminaries}\label{sec:preliminaries}

\subsection{Bridgeland stability conditions}
The main reference for this subsection is \cite{MR2373143}.

\begin{definition}[Slicing on a triangulated category]
    A slicing $\cP$ on a triangulated category $\cD$ is a collection of full additive subcategories $\cP(\phi) \subset \cD$ for $\phi \in \bR$ such that
    \begin{enumerate}
        \item $\cP(\phi + 1) = \cP(\phi)[1]$ for all $\phi \in \bR$.
        \item If $E \in \cP(\phi)$ and $F \in \cP(\phi')$ with $\phi > \phi'$, then $\Hom(E, F) = 0$.
        \item For any $E \in \cD$, there exists a sequence of morphisms
              \begin{equation}
                  0 = E_0 \to E_1 \to \cdots \to E_n = E
              \end{equation}
              in $\cD$ such that the cone $E_{i-1} \to E_i \to F_i \to E_{i-1}[1]$ satisfies $F_i \in \cP(\phi_i)$ for some $\phi_i \in \bR$ and $\phi_1 > \phi_2 > \cdots > \phi_n$.
    \end{enumerate}
\end{definition}
The sequence in the third condition is called the \emph{Harder--Narasimhan (HN) filtration} of $E$, and the objects $F_i$ are called the \emph{Harder--Narasimhan factors} of $E$.
For an interval $I \subset \bR$, we define the full subcategory $\cP(I) \subset \cD$ as the extension closure
\begin{equation}
    \cP(I) = \langle \cP(\phi) \mid \phi \in I \rangle.
\end{equation}
\begin{definition}[Bridgeland stability condition]
    Let $\cD$ be a triangulated category and $v \colon K(\cD) \twoheadrightarrow \Lambda$ be a group homomorphism onto a finite lattice $\Lambda$.
    A stability condition $\sigma = (Z, \cP)$ on $\cD$ (with respect to $v$) consists of:
    \begin{enumerate}
        \item a slicing $\cP$ on $\cD$,
        \item a group homomorphism $Z \colon K(\cD) \to \bC$ which factors through $v$ and satisfies $Z(\cP(\phi) \setminus \{0\}) \subset \bR_{>0} \cdot e^{\sqrt{-1}\pi \phi}$ for all $\phi \in \bR$.
    \end{enumerate}
\end{definition}
The function $Z$ is called the \emph{central charge}. The non-zero (resp. simple) objects of $\cP(\phi)$ are called \emph{$\sigma$-semistable} (resp. \emph{$\sigma$-stable}) of phase $\phi$.
The subcategory consisting of the zero object and $\sigma$-stable objects of phase $\phi$ is denoted by $\cP_{s}(\phi) \subset \cP(\phi)$.
For an object $E$ with the HN factors $F_i$, we define $\phi^+(E) = \max_i \phi(F_i)$ and $\phi^-(E) = \min_i \phi(F_i)$.
For a stability condition $\sigma = (Z, \cP)$, the subcategory $\cA = \cP((0, 1])$ becomes the heart of a bounded t-structure on $\cD$.

\begin{definition}[{\cite{kontsevich2008stabilitystructuresmotivicdonaldsonthomas}}]
    Fix a norm $\| \cdot \|$ on $\Lambda_\bR$.
    We say a stability condition $\sigma = (Z, \cP)$ on $\cD$ (with respect to $v$) satisfies the \emph{support property} if the following condition holds:
    \begin{equation}
        \sup\left\{\frac{\|v(E)\|}{|Z(E)|}\mid E \in \cD \colon \sigma \text{-semistable}\right\} < \infty
    \end{equation}
\end{definition}
The set of stability conditions on $\cD$ with respect to $v \colon K(\cD) \twoheadrightarrow \Lambda$ satisfying the support property is denoted by $\Stab_\Lambda(\cD)$.
The subscript $\Lambda$ may be omitted if $v$ and $\Lambda$ are clear from the context.
It carries a natural topology induced by the following generalized metric:
\begin{equation}
    d(\sigma_1, \sigma_2) = \sup_{0 \neq E \in \cD}\left\{|\phi^+_1(E) - \phi^+_2(E)|, |\phi^-_1(E) - \phi^-_2(E)|, \left|\log \frac{m_1(E)}{m_2(E)}\right|\right\}.
\end{equation}
Here $m_i(E)$ denotes the \emph{mass function} with respect to $\sigma_i$ defined by $m_i(E) = \sum_{j}|Z_i(F_j)|$ where $F_j$ are the HN factors of $E$ with respect to $\sigma_i$, and $\phi_i^{\pm}(E)$ are the phases with respect to $\sigma_i$ defined above.
The following theorem is fundamental in the theory of Bridgeland stability conditions.
\begin{theorem}[{\cite[Theorem 1.2]{MR2373143}}]
    The forgetful map
    \begin{equation}
        \cZ \colon \Stab_\Lambda(\cD) \to \Hom(\Lambda, \bC), \quad (Z, \cP) \mapsto Z
    \end{equation}
    is a local homeomorphism.
    In particular, the space $\Stab_\Lambda(\cD)$ has a natural structure of a complex manifold.
\end{theorem}

There are two natural group actions on the space $\Stab_\Lambda(\cD)$.
The first one is the left action of the autoequivalence group $\Aut(\cD)$ given by
\begin{equation}
    \Phi \cdot (Z, \cP) = (Z \circ \Phi, \{\Phi(\cP(\phi))\}_\phi).
\end{equation}
The second is a continuous right action.
Let $\GL_2^+(\bR)$ be the group of $2 \times 2$ real matrices with positive determinant.
It acts on $\Hom(\Lambda, \bC)$ by $A \colon Z \mapsto A^{-1} \circ Z$, which induces a right action of the universal cover $\widetilde{\GL_2^+}(\bR)$ on $\Stab_\Lambda(\cD)$.
Restricting this action to the subgroup $\bC \subset \widetilde{\GL_2^+}(\bR)$ gives the following explicit description.
\begin{proposition}
    The right action of $\bC$ on $\Stab_\Lambda(\cD)$ is given by
    \begin{equation}
        (Z, \cP) \cdot \lambda = (e^{-\sqrt{-1} \pi \lambda}Z, \{\cP(\phi + \Re \lambda)\}_\phi).
    \end{equation}
\end{proposition}
\subsection{(Half-)Spherical Twists}\label{subsec:half-spherical-twist}
We collect some basic facts about spherical twists from \cite{MR1831820} and half-spherical twists from \cite{arai2024halfspherical}.
Let $\totalSpace$ be a smooth quasi-projective variety of dimension $n$.
\begin{definition}[Spherical object]
    An object $E \in D^b(\totalSpace)$ with proper support is called \emph{($n$-)spherical} if
    \begin{enumerate}
        \item $E \otimes \omega_{\totalSpace} \cong E$, and
        \item $\Ext^*(E, E) \cong \bC \oplus \bC[-n]$.
    \end{enumerate}
\end{definition}
Given a spherical object $E$, we can define a spherical twist $T_E$ as the integral functor $\Phi_{\cP_E}$ with the kernel
\begin{equation}
    \cP_E = \Cone(E^\vee \boxtimes E \to \cO_{\Delta}).
\end{equation}
Here, $\Delta \subset \totalSpace \times \totalSpace$ is the diagonal and the evaluation map $E^\vee \boxtimes E \to \cO_{\Delta}$ is adjoint to the natural pairing $(E^\vee \boxtimes E)\vert_{\Delta} \cong E^\vee \otimes E \to \cO_{\Delta}$.
\begin{theorem}[{\cite[Theorem 1.2]{MR1831820}}]
    The spherical twist $T_E$ gives an autoequivalence of $D^b(\totalSpace)$.
\end{theorem}
The following exact triangle is useful to compute the spherical twist:
\begin{proposition}\label{prop:exact-triangle-spherical-twist}
    For every object $F \in D^b(\totalSpace)$, we have the following exact triangle:
    \begin{equation}
        \Ext^*(E, F) \otimes_{\bC} E \to F \to T_E(F) \xrightarrow{+1}.
    \end{equation}
\end{proposition}

Next, we introduce the notion of half-spherical twists, the ``restricted-to-fiber version'' of spherical twists.
Let $\pi \colon \totalSpace \to T$ be a flat morphism to a smooth quasi-projective $T$ and $\closedImmersion \colon \fiber \hookrightarrow \totalSpace$ denote the inclusion of the (possibly singular) fiber $\fiber = \pi^{-1}(t_0)$ for some $t_0 \in T$.
\begin{definition}[Half-spherical object]
    An object $E \in D^b(\fiber)$ is called \emph{half-spherical} (with respect to $\closedImmersion$) if $\closedImmersion_* E \in D^b(\totalSpace)$ is spherical.
\end{definition}
For a half-spherical object $E \in D^b(\fiber)$, the spherical twist $T_{\closedImmersion_* E}$ is a relative integral functor with respect to $T$.
\begin{proposition}[{\cite[Proposition 4.2]{arai2024halfspherical}}]
    There exists a unique kernel $\cQ_E \in D(\QCoh(\totalSpace \times_T \totalSpace))$ up to isomorphism such that $T_{\closedImmersion_* E} \cong \Phi_{\cQ_E}$.
\end{proposition}
Using this kernel, we define the \emph{half-spherical twist} $\halfTwistE$ as the integral functor $\Phi_{\cQ_E\vert_{\fiber \times \fiber}} \colon D(\QCoh(\fiber)) \to D(\QCoh(\fiber))$.
One can show that $\halfTwistE$ preserves $D^b(\fiber)$, and by \cite[Section 2]{MR2238172} it induces an autoequivalence of $D^b(\fiber)$ that satisfies a certain compatibility with $T_{\closedImmersion_*E}$.
\begin{proposition}[{\cite[Corollary 4.6]{arai2024halfspherical}, cf.~\cite[Theorem 3.6]{MR4830066}}]\label{prop:half-spherical-twist-compatibility}
    The half-spherical twist $\halfTwistE$ is an autoequivalence of $D^b(\fiber)$ which makes the following diagram commutative:
    \begin{equation}
        \begin{tikzcd}
            D^b(\fiber) \ar[r, "\closedImmersion_*"] \ar[d, "\halfTwistE"] & D^b(\totalSpace) \ar[d, "T_{\closedImmersion_* E}"] \ar[r, "\closedImmersion^*"]& D^b(\fiber) \ar[d, "\halfTwistE"]\\
            D^b(\fiber) \ar[r, "\closedImmersion_*"] & D^b(\totalSpace) \ar[r, "\closedImmersion^*"]& D^b(\fiber)
        \end{tikzcd}
    \end{equation}
\end{proposition}

Note that the restriction $D(\QCoh(\totalSpace \times_T \totalSpace)) \to D(\QCoh(\fiber \times \fiber)), \cQ \mapsto \cQ\vert_{\fiber \times \fiber}$ is compatible with convolution of kernels.
In particular, for any relation of autoequivalences on $D(\QCoh(\totalSpace))$ that is satisfied at the level of kernels in $D(\QCoh(\totalSpace \times_T \totalSpace))$, one has a corresponding relation on $D(\QCoh(\fiber))$ by restricting the kernels to $\fiber \times \fiber$.
\begin{example}\label{ex:half-spherical-twist-conjugation}
    Let $\widetilde{f}$ be an automorphism of $\totalSpace$ over $T$ and $f = \widetilde{f}\vert_{\fiber}$ be its restriction to $\fiber$.
    For a half-spherical object $E \in D^b(\fiber)$ we have
    \begin{equation}\label{eq:spherical-twist-conjugation}
        \widetilde{f}_* \circ T_{\closedImmersion_* E} \circ (\widetilde{f}_*)^{-1} \cong T_{\widetilde{f}_* \closedImmersion_* E} \cong T_{\closedImmersion_* f_* E}
    \end{equation}
    by \cite[Lemma 2.11]{MR1831820}.
    In particular, $f_* E$ is also a half-spherical object.
    Since the kernel of $\widetilde{f}$ is defined on $\totalSpace \times_T \totalSpace$ and the relation \eqref{eq:spherical-twist-conjugation} holds at the level of kernels, one has
    \begin{equation}
        f_* \circ \halfTwistE \circ (f_*)^{-1} \cong H_{f_* E}.
    \end{equation}
\end{example}

\section{Setting}\label{sec:setting}
The notation in this section will be used throughout the paper.
Consider the $\mu_3$-action on $\bA^3$ given by
\begin{equation}
    \omega \cdot (x_1, x_2, x_3) = (\omega x_1, \omega x_2, \omega x_3), \quad \omega^3 = 1.
\end{equation}
Let
$\totalSpace = \Tot(K_{\canonicalBundleBaseSpace}) \to \bA^3/\mu_3$
be a resolution of the quotient singularity $\bA^3/\mu_3$.
Composing this morphism with
\begin{equation}
    \bA^3/\mu_3 \to \baseSpace, \quad [(x_1, x_2, x_3)] \mapsto x_1x_2x_3
\end{equation}
gives a map $\projectionMap \colon \totalSpace \to \baseSpace$.
The central fiber $\closedImmersion \colon \fiber = \projectionMap^{-1}(0) \hookrightarrow \totalSpace$ consists of three copies of $\Bl_0\bA^2$ and one $\bP^2$, which we denote by $\component_1, \component_2, \component_3$ and $\cptComponent$:
\begin{equation}
    X = \cptComponent \cup \component_1 \cup \component_2 \cup \component_3.
\end{equation}
These components meet along the exceptional divisors $E_i = \cptComponent \cap \component_i \cong \bP^1$ of the blow-ups $\component_i \to \bA^2$ for $i = 1, 2, 3$.
The compact component $\cptComponent \cong \bP^2$ is both the exceptional divisor of the resolution $\totalSpace \to \bA^3/\mu_3$ and the image of the zero section $\sectionMap$ of the line bundle $\totalSpace = \Tot(K_{\canonicalBundleBaseSpace}) \to \canonicalBundleBaseSpace$.
We denote the inclusion $\cptComponent \hookrightarrow \fiber$ by $\cptComponentImmersion$ and set $j = \closedImmersion \circ \cptComponentImmersion$.
The setup is summarized in the following diagram:
\begin{equation}
    \begin{tikzcd}
        &&\canonicalBundleBaseSpace \ar[lld, bend right, "\sectionMap"'] \\
        \cptComponent \ar[r, hook, "\cptComponentImmersion"]& X \ar[d]\ar[r, hook, "\closedImmersion"] & \totalSpace = \Tot(K_{\canonicalBundleBaseSpace}) \ar[d, "\projectionMap"] \ar[u, "\canonicalBundleProjection"]\\
        & \{0\} \ar[r]& \baseSpace.
    \end{tikzcd}
\end{equation}
Observe that the map $\projectionMap \colon \totalSpace \to \baseSpace$ is a toric morphism and $\fiber$ is the toric boundary of $\totalSpace$.
The canonical bundle $\omega_{\totalSpace}$ is trivial, i.e.~$\totalSpace$ is a (non-compact) Calabi--Yau.
Since $\fiber$ is the toric boundary, its canonical bundle $\omega_{\fiber}$ is also trivial.

The derived category of coherent sheaves on $\fiber$ is denoted by $\fiberD = D^b(\fiber)$.
Its full subcategory consisting of objects with support contained in $\cptComponent$ is denoted by $\fiberDsupp = D^b_{\cptComponent}(\fiber)$.
It can be viewed as the derived category of the abelian category $\fiberCohSupp = \Coh_{\cptComponent}(\fiber)$ of coherent sheaves on $\fiber$ supported on $\cptComponent$.
\begin{remark}\label{rem:compactification-to-type-III-degeneration}

    Although it is not directly needed in this paper, a mirror partner of the family $\totalSpace$ is given by $M = (\bC^*)^2$ with the divisor $D = \{z_1 + z_2 + c/(z_1z_2) + 1 = 0\} \subset M$, where $c$ is a suitable constant \cite{MR3838112}.
    Indeed, applying \cite[Theorem 1.1]{MR4554224} to the toric variety $\totalSpace = \Tot(K_{\canonicalBundleBaseSpace})$ and the Laurent polynomial $W = z_0(z_1 + z_2 + c/(z_1z_2) + 1) + 1$ on $(\bC^*)^3$
    yields a derived equivalence between the wrapped Fukaya category of $M \setminus D$ and the category of coherent sheaves on $\fiber$.
    See also \cite[Theorem~3.4.4]{MR4845526} and the remarks that follow.
\end{remark}

In the rest of this section, we determine the Picard group and automorphism group of $\fiber$.
Denote by $Z_{ij}$ the intersection $\component_i \cap \component_j$ for $i, j = 1, 2, 3$ with $i \neq j$.
They are isomorphic to $\bA^1$.

\begin{lemma}\label{lem:picard-group-of-non-compact-part}
    We have an isomorphism
    \begin{equation}
        \Pic(\component_1 \cup \component_2 \cup \component_3) \cong \bZ^3 \times \bC^*.
    \end{equation}
\end{lemma}
\begin{proof}
    A line bundle on $\component_1 \cup \component_2 \cup \component_3$ is determined by the following data:
    \begin{enumerate}
        \item a line bundle $\cL_i$ on $\component_i$ for $i = 1, 2, 3$,
        \item gluing isomorphisms $\varphi_{ji} \colon \cL_i \vert_{Z_{ij}} \xrightarrow{\sim} \cL_j \vert_{Z_{ij}}$.
    \end{enumerate}
    A line bundle on a component $\component_i \cong \Bl_0 \bA^2$ is of the form $\cO_{\component_i}(n_i E_i)$ for some $n_i \in \bZ$, where $E_i$ is the exceptional divisor of the blow-up.
    Moreover, since every line bundle on $Z_{ij} \cong \bA^1$ is trivial, we fix a trivialization $\cO_{\component_i}(n_i E_i) \vert_{Z_{ij}} \cong \cO_{Z_{ij}}$ for each $i \neq j$.
    Their automorphisms are parametrized by $\Gamma(Z_{ij}, \cO_{Z_{ij}}^*) = \bC^*$.
    With these choices, the gluing data is equivalent to:
    \begin{enumerate}
        \item a triple of integers $(n_1, n_2, n_3)$, and
        \item a triple of non-zero complex numbers $(\varphi_{21}, \varphi_{32}, \varphi_{13})$.
    \end{enumerate}
    Two such data $(n_1, n_2, n_3, \varphi_{21}, \varphi_{32}, \varphi_{13})$ and $(n'_1, n'_2, n'_3, \varphi'_{21}, \varphi'_{32}, \varphi'_{13})$ define the same line bundle if and only if $(n_1, n_2, n_3) = (n'_1, n'_2, n'_3)$ and there exist automorphisms $\psi_i$ of $\cO_{\component_i}(n_i E_i)$ such that
    \begin{equation}\label{eq:coboundary-condition}
        \psi_j \cdot \varphi_{ji} = \varphi'_{ji} \cdot \psi_i.
    \end{equation}
    Since $\Aut(\cO_{\component_i}(n_i E_i)) = \Gamma(\component_i, \cO_{\component_i}^*) = \bC^*$, the relations \eqref{eq:coboundary-condition} can be viewed in $\bC^*$.
    This condition is equivalent to $\varphi_{21}\varphi_{32}\varphi_{13} = \varphi'_{21}\varphi'_{32}\varphi'_{13}$.
    Therefore, the surjection
    \begin{equation}
        \{(n_1, n_2, n_3, \varphi_{21}, \varphi_{32}, \varphi_{13}) \in \bZ^3 \times (\bC^*)^3\} \twoheadrightarrow \Pic(\component_1 \cup \component_2 \cup \component_3)
    \end{equation}
    factors through $\bZ^3 \times \bC^*$, establishing the claimed isomorphism.
\end{proof}
\begin{proposition}\label{prop:picard-group-of-fiber}
    The restriction $\Pic(\totalSpace) \to \Pic(\fiber)$ is an isomorphism.
    In particular, we have $\Pic(\fiber) \cong \bZ$ with a generator $\cO_{\fiber}(1) = \canonicalBundleProjection^*\cO_{\canonicalBundleBaseSpace}(1) \vert_{\fiber}$.
\end{proposition}
\begin{proof}
    Since $\Pic(\totalSpace) \to \Pic(\cptComponent)$ is an isomorphism, it is enough to show that $\Pic(\fiber) \to \Pic(\cptComponent)$ is injective.

    Denote $\component_1 \cup \component_2 \cup \component_3$ by $W$.
    The Picard group of the cycle of three projective lines $\cptComponent \cap W = E_1 \cup E_2 \cup E_3$ is
    \begin{equation}
        \Pic(\cptComponent \cap W) \cong \bZ^3 \times \bC^*
    \end{equation}
    with $\bZ^3$-part representing the degree vector, and $\bC^*$-part representing the monodromy.
    Then by Lemma \ref{lem:picard-group-of-non-compact-part} (and its proof) the restriction map $\Pic(W) \to \Pic(\cptComponent \cap W)$ is an isomorphism.

    Suppose we are given $\cL \in \Pic(\fiber)$ such that $\cL\vert_{\cptComponent} = \cO_{\cptComponent}$.
    The exact sequence $0 \to \cO_{\fiber} \to \cO_{\cptComponent} \oplus \cO_W \to \cO_{\cptComponent \cap W} \to 0$ induces an exact sequence
    \begin{equation}
        0 \to \cL \to \cL\vert_{\cptComponent} \oplus \cL\vert_{W} \to \cL\vert_{\cptComponent \cap W} \to 0.
    \end{equation}
    One has $\cL\vert_{\cptComponent \cap W} = (\cL\vert_{\cptComponent})\vert_{\cptComponent \cap W} = \cO_{\cptComponent \cap W}$ and hence $\cL\vert_W = \cO_W$ by $\Pic(W) \xrightarrow{\sim} \Pic(\cptComponent \cap W)$.
    This shows $\cL = \cO_{\fiber}$.
\end{proof}

Let $x_{ij}$ be the unique intersection point of $Z_{ij}$ and $\cptComponent$.
Note that every automorphism of $\fiber$ induces an automorphism of $Z_{12} \cup Z_{23} \cup Z_{31}$ and preserves the set $\{x_{12}, x_{23}, x_{31}\}$, hence we have
\begin{equation}\label{eq:restriction-map}
    \Aut(\fiber) \to \Aut_{\{x_{12}, x_{23}, x_{31}\}}(Z_{12} \cup Z_{23} \cup Z_{31}).
\end{equation}
The latter is isomorphic to $S_3 \ltimes (\bC^*)^3$, where $S_3$ permutes the components and $(\bC^*)^3$ acts on $Z_{ij} = \bA^1$ by scaling with origin $x_{ij}$ fixed.
\begin{proposition}\label{prop:automorphism-group-of-fiber}
    The restriction map \eqref{eq:restriction-map} is an isomorphism.
    In particular, $\Aut(\fiber) \cong S_3 \ltimes (\bC^*)^3$.
\end{proposition}
\begin{proof}
    First, fix $\lambda = (\lambda_{12}, \lambda_{23}, \lambda_{31}) \in (\bC^*)^3 \subset \Aut_{\{x_{12}, x_{23}, x_{31}\}}(Z_{12} \cup Z_{23} \cup Z_{31})$.
    We claim that there is a unique lift $f_\lambda \in \Aut(\fiber)$ such that
    \begin{equation}
        f_\lambda\vert_{Z_{ij}} = \lambda_{ij} \in \bC^* = \Aut_{x_{ij}}(Z_{ij}).
    \end{equation}
    These conditions determine automorphisms of each component $\component_i \cong \Bl_0\bA^2$, and they glue to an automorphism of $\component_1 \cup \component_2 \cup \component_3$.
    Moreover, this glued automorphism extends uniquely to an automorphism of $\fiber$, since an automorphism of $\cptComponent = \bP^2$ is determined by the scalings on the toric boundary $\cptComponent \cap (\component_1 \cup \component_2 \cup \component_3) = E_1 \cup E_2 \cup E_3$.
    Conversely, any automorphism of $\fiber$ that does not permute the components $\component_1, \component_2, \component_3$ arises in this way.

    Next, choose an embedding $S_3 \hookrightarrow \Aut(\fiber)$, $\sigma \mapsto f_\sigma$, so that $f_\sigma$ permutes the components $\component_1, \component_2, \component_3$ according to $\sigma$.
    These two types of automorphisms generate $\Aut(\fiber)$, and the induced action of $S_3$ on $(\bC^*)^3$ is by permuting the coordinates.
\end{proof}
\begin{proposition}\label{prop:automorphism-of-fiber-extends-to-total-space}
    Any automorphism of $\fiber$ extends to an automorphism of $\totalSpace$.
    Moreover, any automorphism in $S_3 \ltimes \{(\lambda_1, \lambda_2, \lambda_3) \in (\bC^*)^3 \mid \lambda_1 \lambda_2 \lambda_3 = 1\} \subset S_3 \ltimes (\bC^*)^3 \cong \Aut(\fiber)$ extends to an automorphism of $\totalSpace$ over $\baseSpace$.
\end{proposition}
\begin{proof}
    For the first part, it suffices to show that any element of $\Aut_{\{x_{12}, x_{23}, x_{31}\}}(Z_{12} \cup Z_{23} \cup Z_{31})$ extends to an automorphism of $\totalSpace$ by Proposition \ref{prop:automorphism-group-of-fiber}.
    The $S_3$-part is straightforward: the permutation of variables
    \begin{equation}
        \bA^3/\mu_3 \to \bA^3/\mu_3, \quad (x_1, x_2, x_3) \mapsto (x_{\sigma(1)}, x_{\sigma(2)}, x_{\sigma(3)})
    \end{equation}
    induces the desired automorphisms of $\totalSpace$ via the blow-up $\totalSpace \to \bA^3/\mu_3$.

    For the $(\bC^*)^3$-part, take $\mu = (\mu_1, \mu_2, \mu_3) \in (\bC^*)^3$, and let $f_\mu$ denote the automorphism of $\totalSpace$ induced by
    \begin{equation}
        \overline{f_\mu} \colon \bA^3/\mu_3 \to \bA^3/\mu_3, \quad (x_1, x_2, x_3) \mapsto (\mu_1 x_1, \mu_2 x_2, \mu_3 x_3).
    \end{equation}
    When restricted to the line
    \begin{align}
        \{x_1 = x_2 = 0\} & = \Spec \bC[x_1, x_2, x_3]^{\mu_3}/(x_1, x_2) \\
                          & = \Spec \bC[x_3]^{\mu_3}                      \\
                          & = \Spec \bC[x_3^3]
    \end{align}
    the automorphism $\overline{f_\mu}$ acts by $x_3^3 \mapsto \mu_3^3 x_3^3$.
    Since $Z_{12}$ is the strict transform of this line, the induced automorphism $f_\mu \vert_{Z_{12}}$ must be $\mu_3^3 \in \bC^* = \Aut_{x_{12}}(Z_{12})$.
    In particular, any element $(\lambda_1, \lambda_2, \lambda_3) \in (\bC^*)^3 \subset \Aut_{\{x_{12}, x_{23}, x_{31}\}}(Z_{12} \cup Z_{23} \cup Z_{31})$ is a restriction of $f_\mu$ with $\mu_i = \lambda_i^{1/3}$.
    This shows the first part of the statement.

    Moreover, if $\lambda_1 \lambda_2 \lambda_3 = 1$, we can take $\mu$ so that $\mu_i^3 = \lambda_i$ for all $i$ and $\mu_1 \mu_2 \mu_3 = 1$.
    Then the second part of the statement follows.
\end{proof}

\begin{remark}\label{rem:diagonal-automorphism-acts-trivially}
    An automorphism $f \in \Aut(\totalSpace)$ contained in the subgroup $\{(\lambda, \lambda, \lambda) \mid \lambda \in \bC^*\}$ of the $(\bC^*)^3$-part of $\Aut(\fiber)$ acts on $\cptComponent$ trivially.
\end{remark}

\section{Various Derived Categories}\label{sec:various-derived-categories}
In this section, we use the following notation:
\begin{enumerate}
    \item $\fiberDc = D^b_c(\fiber) \subset \fiberD$: the full subcategory of complexes with proper support.
    \item $\Perf(\fiber)\subset \fiberD$: the full subcategory of perfect complexes.
\end{enumerate}

\begin{lemma}\label{lem:decomposition-of-proper-subset}
    Let $W \subset \fiber$ be a closed subset which is proper over the base field (when regarded as a reduced closed subscheme).
    Then $W$ is a finite union of
    \begin{enumerate}
        \item closed subsets of $\cptComponent$,
        \item closed points in $\fiber \setminus \cptComponent$.
    \end{enumerate}
\end{lemma}
\begin{proof}
    Since each $\component_i$ (with $i = 1, 2, 3$) is isomorphic to $\Bl_0\bA^2$, its closed subscheme which is proper over the base field must be a finite union of closed points and the exceptional divisor $E_i = \component_i \cap \cptComponent$ of the blow-up.
    Thus, $W$ must be a finite union of closed subschemes of $\cptComponent$ (including the exceptional divisors), and points in $\component_i \setminus \cptComponent$ for $i = 1, 2, 3$.
\end{proof}

\begin{proposition}\label{prop:decomposition-of-compact-support-complex}
    Let $\cD' \subset \fiberDc$ be the full subcategory of complexes supported on $\fiber \setminus \cptComponent$.
    Then the category $\fiberDc$ is decomposed as
    \begin{equation}
        \fiberDc = \fiberDsupp \oplus \cD'.
    \end{equation}
\end{proposition}
\begin{proof}
    Let $F \in \fiberDc$ be a complex with proper support.
    By Lemma \ref{lem:decomposition-of-proper-subset}, the support $\Supp(F)$ is a finite union of closed subsets of $\cptComponent$ and closed points in $\fiber \setminus \cptComponent$.
    Therefore, there exists a unique decomposition
    \begin{equation}\label{eq:decomposition-of-compact-support-complex}
        F = F_Z \oplus \bigoplus_{x \in S} F_x
    \end{equation}
    where $F_Z$ is a complex supported on $\cptComponent$, $S$ is a finite set of points in $\fiber \setminus \cptComponent$, and $F_x$ is a complex supported on $x$.
    Since this decomposition is clearly functorial and compatible with exact triangles, the statement follows.
\end{proof}

\begin{remark}\label{rem:decomposition-of-compact-support-complex}
    \begin{enumerate}
        \item The objects of the category $\cD'$ are of the form $\bigoplus_{x \in S} F_x$ where $S$ is a finite set of points in $\fiber \setminus \cptComponent$ and $F_x$ is a complex supported on $x$.
        \item In particular, every indecomposable object of $\fiberDc = \fiberDsupp \oplus \cD'$ is supported on either $\cptComponent$ or a single point $x \in \fiber \setminus \cptComponent$.
    \end{enumerate}
\end{remark}

The sequence of natural functors
\begin{equation}
    D^b(\cptComponent) \xrightarrow{\cptComponentImmersion_*} \fiberDsupp \hookrightarrow \fiberDc \hookrightarrow \fiberD
\end{equation}
induces natural maps
\begin{equation}\label{eq:K-groups}
    K(\cptComponent)=K(D^b(\cptComponent)) \to K(\fiberDsupp) \to K(\fiberDc) \to K(\fiberD).
\end{equation}

\begin{proposition}\label{prop:K-groups}
    The morphisms \eqref{eq:K-groups} induce the following identifications:
    \begin{enumerate}
        \item $K(\fiberDsupp) = K(\cptComponent)$.
        \item $K(\fiberDc) = K(\cptComponent) \oplus \bigoplus_{x \in \fiber \setminus \cptComponent}\bZ[\cO_x]$.

    \end{enumerate}
\end{proposition}
\begin{proof}
    For (1), see \cite[Example II 6.3.4]{MR3076731}.
    For (2), the decomposition in Proposition \ref{prop:decomposition-of-compact-support-complex} induces a decomposition of the Grothendieck group
    \begin{equation}
        K(\fiberDc) = K(\fiberDsupp) \oplus K(\cD').
    \end{equation}
    The group $K(\fiberDsupp)$ is isomorphic to $K(\cptComponent)$ by (1).
    Since every complex in $\cD'$ is supported on a finite set of points, the group $K(\cD')$ is isomorphic to $\bigoplus_{x \in \fiber \setminus \cptComponent}\bZ[\cO_x]$.
\end{proof}

\begin{proposition}\label{prop:auto-preserves-subcategories}
    Any autoequivalence $\Phi \in \Aut(\fiberD)$ preserves the subcategories $\fiberDc$ and $\fiberDsupp$.
\end{proposition}
\begin{proof}
    The statement follows from the following two lemmas.
\end{proof}

\begin{lemma}\label{lem:auto-preserves-compact-support}
    Any autoequivalence $\Phi \in \Aut(\fiberD)$ preserves the subcategory $\fiberDc$.
\end{lemma}
\begin{proof}
    By \cite[Proposition 1.11]{MR2437083} the subcategory $\fiberPerf$ of perfect complexes can be reconstructed as the full-subcategory of \emph{homologically finite objects} \cite[Definition 1.6]{MR2437083} in $\fiberD$ only from its triangulated category structure.
    In particular, $\Phi$ preserves $\fiberPerf$.

    Next, we show that $\Phi$ preserves $\fiberDc$.
    Every complex $E \in D(\QCoh(\fiber))$ defines a cohomological functor
    \begin{equation}
        \Hom(-, E) \colon \fiberPerf^{\op} \to \Vect_{\bC}.
    \end{equation}
    Then \cite[Theorem 3.2]{MR2793026} says that a complex $E \in D(\QCoh(\fiber))$ is contained in $\fiberDc$ if and only if the functor $\Hom(-, E)$ is \emph{locally finite} in the sense of \cite[Definition 3.1]{MR2793026}.
    By using this characterization, for any $E \in \fiberD$ we have
    \begin{align}
             & \Phi(E) \in \fiberDc                                                                                     \\
        \iff & \Hom(-, \Phi(E)) \cong \Hom(\Phi^{-1}(-), E) \colon \fiberPerf \to \Vect_{\bC} \text{ is locally finite} \\
        \iff & \Hom(-, E) \colon \fiberPerf \to \Vect_{\bC}\text{ is locally finite}                                    \\
        \iff & E \in \cD_c.
    \end{align}
    We used $\Phi(\fiberPerf) = \fiberPerf$ in the second equivalence.
\end{proof}

\begin{lemma}
    Any autoequivalence $\Phi \in \Aut(\fiberDc)$ respects the decomposition $\fiberDc = \fiberDsupp \oplus \cD'$ in Proposition \ref{prop:decomposition-of-compact-support-complex}, i.e.~one has $\Phi(\fiberDsupp) = \fiberDsupp$ and $\Phi(\cD') = \cD'$.
\end{lemma}
\begin{proof}
    We first show that the induced automorphism $\Phi \colon K(\fiberDc) \to K(\fiberDc)$ preserves the subgroup $K(\fiberDsupp)$.
    The group $K(\fiberDsupp) = K(\cptComponent)$ is generated by $[\cO_{\cptComponent}]$, $[\cO_l]$, and $[\cO_x]$, where $l \subset \cptComponent$ is a line and $x \in \cptComponent$ is a point.
    The objects $\Phi(\cO_{\cptComponent})$, $\Phi(\cO_l)$, and $\Phi(\cO_x)$ are indecomposable and there are non-zero morphisms $\Phi(\cO_{\cptComponent}) \to \Phi(\cO_l)$ and $\Phi(\cO_{\cptComponent}) \to \Phi(\cO_x)$.
    Therefore, Proposition \ref{prop:decomposition-of-compact-support-complex} and Remark \ref{rem:decomposition-of-compact-support-complex} (2) imply that there are only two possibilities:
    \begin{enumerate}
        \item $\Phi(\cO_{\cptComponent}), \Phi(\cO_l), \Phi(\cO_x)$ are contained in $\fiberDsupp$, or
        \item $\Phi(\cO_{\cptComponent}), \Phi(\cO_l), \Phi(\cO_x)$ are supported on a common point $y \in \fiber \setminus \cptComponent$.
    \end{enumerate}
    The second case does not occur; otherwise, the rank three lattice $\Phi(K(\fiberDsupp))$ would be contained in $\bZ[\cO_y]$.
    This shows $\Phi(K(\fiberDsupp)) = K(\fiberDsupp)$.

    Next, we show $\Phi(\cD') \subset \cD'$.
    If not, there exists a point $y \in \fiber \setminus \cptComponent$ with $\Phi(\cO_y) \notin \cD'$.
    Since $\Phi(\cO_y)$ is indecomposable, we have $\Phi(\cO_y) \in \fiberDsupp$.
    Passing to the Grothendieck group, we have $0 \neq [\cO_y] = [\Phi^{-1}(\Phi(\cO_y))] \in \Phi^{-1}(K(\fiberDsupp)) = K(\fiberDsupp)$, which is a contradiction.
    Therefore, the autoequivalence $\Phi$ preserves the subcategory $\cD'$.
    It also preserves the perpendicular category $\fiberDsupp = {}^\perp\cD'$.
\end{proof}

\begin{corollary}\label{cor:outside-skyscraper-sheaf}
    For any autoequivalence $\Phi \in \Aut(\fiberDc)$ and $x \in \fiber \setminus \cptComponent$, the image $\Phi(\cO_x)$ of the skyscraper sheaf $\cO_x$ is of the form $\cO_y[n]$ for some point $y \in \fiber \setminus \cptComponent$ and integer $n \in \bZ$.
\end{corollary}
\begin{proof}
    By the indecomposability of $\cO_x$ and the previous lemma, the image $\Phi(\cO_x)$ is supported on a single point $y \in \fiber \setminus \cptComponent$.
    It also satisfies the following condition:
    \begin{equation}
        \Ext^i(\Phi(\cO_x), \Phi(\cO_x)) = \begin{cases}
            \bC & i = 0, \\
            0   & i < 0
        \end{cases}
    \end{equation}
    Then the argument in \cite[Lemma 4.5]{MR2244106} shows the claim.
\end{proof}

\section{Stability conditions}\label{sec:stability-conditions}
In this section, we study the space of Bridgeland stability conditions on the category $\fiberDsupp$ by comparing it with that on $D^b_{\cptComponent}(\totalSpace)$.
For simplicity, we denote the category $D^b_{\cptComponent}(\totalSpace)$ by $\simpleTotalDsupp$.
We say that a vector bundle $E$ on $\cptComponent$ is \emph{exceptional} if $\Ext^i_{\cptComponent}(E, E) = \bC$ for $i = 0$ and $0$ for $i \neq 0$.
If $E$ is exceptional, then the sheaf $\closedImmersion_*\cptComponentImmersion_*E$ on $\totalSpace$ is a spherical object in $\simpleTotalDsupp$ \cite[Proposition 3.15]{MR1831820} and hence $\cptComponentImmersion_*E$ is a half-spherical object in $\fiberDsupp$.
By abuse of notation, we denote the associated spherical twist by $\sphericalTwistE \colon \simpleTotalDsupp \to \simpleTotalDsupp$ and the half-spherical twist by $\halfTwistE \colon \fiberDsupp \to \fiberDsupp$.

\subsection{Comparison isomorphism of spaces of stability conditions}\label{subsec:comparison-isomorphism}
\begin{lemma}
    For any stability condition $\sigma = (Z, \cP)$ on $\simpleTotalDsupp$, the following construction gives a stability condition $\comparisonIso(\sigma) = (Z', \cP')$ on $\fiberDsupp$:
    \begin{enumerate}
        \item $Z' = Z \circ \closedImmersion_*$,
        \item $\cP'(\phi) = \{E \in \fiberDsupp \mid \closedImmersion_* E \in \cP(\phi)\}$.
    \end{enumerate}
\end{lemma}
\begin{proof}
    We can apply \cite[Corollary 2.2.2]{MR2324559} to our supported derived categories.
    The assumption that the functor $-\otimes_{\totalSpace}\closedImmersion_*\cO_{\fiber} \colon \simpleTotalDsupp \to \simpleTotalDsupp$ is right t-exact is satisfied since
    $\closedImmersion_*\cO_{\fiber} \cong (0 \to \cO_{\totalSpace} \to \cO_{\totalSpace} \to 0)$ in $D^b(\totalSpace)$.
\end{proof}
\begin{remark}
    In the above lemma, we also have $\cP'(I) = \{E \in \fiberDsupp \mid \closedImmersion_* E \in \cP(I)\}$ for any interval $I \subset \bR$.
    See also \cite[Proposition 2.2.1]{MR2324559}.
\end{remark}
\begin{lemma}[{\cite[Lemma 5.3]{MR2448280}}]\label{lem:lemma-on-inducing-stability-condition}
    Let $\sigma = (Z, \cP)$ be a stability condition on $\simpleTotalDsupp$ and $\theta(\sigma) = (Z', \cP')$ be the induced stability condition on $\fiberDsupp$.
    Denote by $\cA = \cP((0, 1])$ and $\cA' = \cP'((0, 1])$ the hearts of $\sigma$ and $\theta(\sigma)$ respectively.
    Then the following hold:
    \begin{enumerate}
        \item $\closedImmersion_* \colon \cA' \to \cA$ is fully faithful.
        \item $\closedImmersion_*\cP'_{s}(\phi) = \cP_{s}(\phi)$ for all $\phi \in \bR$, where $\cP_s(\phi)$ (resp.~$\cP'_s(\phi)$) is the subcategory of $\sigma$-stable (resp.~$\theta(\sigma)$-stable) objects of phase $\phi$.
        \item For any $E \in \fiberDsupp$, we have $\phi_{\sigma}^{\pm}(\closedImmersion_*E) = \phi_{\theta(\sigma)}^{\pm}(E)$.
        \item We have \begin{align}
                   & \sup\left\{\frac{\|E\|}{|Z(E)|}\mid E \in \simpleTotalDsupp \colon \sigma \text{-semistable}\right\}      \\
                   & = \sup\left\{\frac{\|E\|}{|Z'(E)|}\mid E \in \fiberDsupp \colon \theta(\sigma) \text{-semistable}\right\}
              \end{align}
              for any norm $\|\cdot\|$ on $K(\simpleTotalDsupp) \cong K(\fiberDsupp)$.
              In particular, $\sigma$ satisfies the support property if and only if $\theta(\sigma)$ does.
    \end{enumerate}
\end{lemma}
These two lemmas show that $\theta$ defines a map from $\Stab(\simpleTotalDsupp)$ to $\Stab(\fiberDsupp)$ which makes the following diagram commute:
\begin{equation}
    \begin{tikzcd}
        \Stab(\simpleTotalDsupp) \ar[r, "\comparisonIso"] \ar[d, "\cZ"'] & \Stab(\fiberDsupp) \ar[d, "\cZ"] \\
        \Hom(K(\simpleTotalDsupp), \bC) \ar[r, "- \circ \closedImmersion_*"] & \Hom(K(\fiberDsupp), \bC).
    \end{tikzcd}
\end{equation}
Notice that the map $\comparisonIso$ commutes with the natural $\bC$-actions on both sides.
By the same argument as in \cite[Corollary 5.4, Theorem 5.5]{MR2448280}, it induces the following comparison isomorphism of the spaces of stability conditions.
\begin{theorem}\label{thm:comparison-isomorphism}
    The map $\theta$ is continuous.
    Moreover, let $\Stab^{\circ}(\simpleTotalDsupp)$ be a connected component of $\Stab(\simpleTotalDsupp)$ and $\Stab^{\circ}(\fiberDsupp)$ be the connected component of $\Stab(\fiberDsupp)$ containing $\theta(\Stab^{\circ}(\simpleTotalDsupp))$.
    Then the map $\comparisonIso \colon \Stab^{\circ}(\simpleTotalDsupp) \to \Stab^{\circ}(\fiberDsupp)$ is an isomorphism of complex manifolds.
\end{theorem}

\subsection{Wall and chamber structure on $\Stab(\simpleTotalDsupp)$}\label{subsec:wall-and-chamber-structure-on-total-space}
Let us recall the structure of $\Stab(\simpleTotalDsupp)$ studied in \cite{bayer-macri}.

A stability condition $\sigma = (Z, \cP) \in \Stab(\simpleTotalDsupp)$ is called \emph{geometric} if the skyscraper sheaves $\cO_x$ for all $x \in \cptComponent$ are $\sigma$-stable of the same phase.
Denote the set of all geometric stability conditions on $\simpleTotalDsupp$ by $\totalGeometricChamber$.
It is an open subset of $\Stab(\simpleTotalDsupp)$ and its boundary $\partial \totalGeometricChamber$ is a locally finite union of \emph{walls} (real codimension one submanifolds) in $\Stab(\simpleTotalDsupp)$.

Define the natural $\bC$-slice $\nTotalGeometricChamber$ of $\totalGeometricChamber$ by
\begin{equation}
    \nTotalGeometricChamber = \{\sigma = (Z, \cP) \in \totalGeometricChamber \mid Z(\cO_x) = -1, \phi(\cO_x) = 1 \text{ for all } x \in \cptComponent\}.
\end{equation}
Consider a basis $([\cO_x], [\cO_l], [\cO_{\cptComponent}])$ of $K(\simpleTotalDsupp)$ where $x \in \cptComponent$ is a point and $l \subset \cptComponent$ is a line.
Denote the corresponding dual basis of $\Hom(K(\simpleTotalDsupp), \bC)$ by $(c, d, r)$:
\begin{equation}
    [F] = c(F)[\cO_x] + d(F)[\cO_l] + r(F)[\cO_{\cptComponent}].
\end{equation}
By definition, the central charge $Z$ of any $\sigma \in \nTotalGeometricChamber$ is of the form $Z = Z_{a, b} = -c + a d + b r$ for some $(a, b) \in \bC^2$.

For a purely 2-dimensional sheaf $F$ on $\totalSpace$ supported on $\cptComponent$, the slope $\slopeOfF$ is defined by $\slopeOfF = d(F)/r(F)$, and slope-stability is defined in the usual way.
One has the Harder--Narasimhan filtration with respect to slope-stability for such sheaves.
Every exceptional vector bundle $E$ on $\cptComponent$ gives rise to a slope-stable sheaf $\cptComponentImmersion_*E$.
For $B \in \bR$, we define a torsion pair $(\widetilde{\cT_B}, \widetilde{\cF_B})$ of $\Coh_{\cptComponent}(\totalSpace) \subset \simpleTotalDsupp$ as follows:
$\widetilde{\cT_B}$ is the extension-closed subcategory generated by slope-semistable sheaves $F$ with $\slopeOfF > B$ and torsion sheaves, and $\widetilde{\cF_B}$ is the extension-closed subcategory generated by slope-semistable sheaves $F$ with $\slopeOfF \leq B$.
The tilt $\langle \widetilde{\cF_B}[1], \widetilde{\cT_B} \rangle$ gives the heart of a bounded t-structure on $\simpleTotalDsupp$.
\begin{lemma}[{\cite[Section 3]{bayer-macri}}]
    Any $\sigma \in \nTotalGeometricChamber$ is of the form $\sigma_{a, b} = (Z_{a, b}, \langle \widetilde{\cF_B}[1], \widetilde{\cT_B} \rangle)$ for some $(a, b) \in \bC^2$ with $\Im a > 0$ and $B = B(a, b) = -\frac{\Im b}{\Im a}$.
\end{lemma}
This shows that the map $(Z, \cP) \mapsto Z$ is an injection from $\nTotalGeometricChamber$ to $\{Z_{a, b} \mid (a, b) \in \bC^2, \Im a > 0\}$.
Next we describe the image.
For $\sigma = (Z_{a, b}, \langle \widetilde{\cF_B}[1], \widetilde{\cT_B} \rangle)$,
suppose there exists a slope-stable sheaf $F$ with slope $\mu(F) = B(a, b) = -\frac{\Im b}{\Im a}$ and discriminant $\discriminantOfF \coloneq \frac{d(F)^2}{2r(F)^2} - \frac{c(F)}{r(F)}$.
Then $F[1]$ lies in $\cP((0, 1]) = \langle \widetilde{\cF_B}[1], \widetilde{\cT_B} \rangle$ and satisfies $\Im Z(F[1]) = -\Im a \cdot r(F)(\slopeOfF - B(a, b)) = 0$, and hence
\begin{equation}
    \Re Z(F[1]) = r(F) \cdot(g(a, b) - \Delta_F) < 0
\end{equation}
where
\begin{equation}\label{eq:g}
    g(a, b) =  -\Re b - B(a, b) \cdot \Re a + \frac{1}{2}B(a,b)^2.
\end{equation}
This suggests that the point $f(a, b) = (B(a, b), g(a, b)) \in \bR^2$ should lie below the set
\begin{equation}
    S = \{(\slopeOfF, \discriminantOfF) \mid F \text{ is a slope-stable sheaf on }\cptComponent\}.
\end{equation}

On the other hand, according to \cite{MR816365, MR1428426} the set $S$ has the following explicit description.
First, for a given $B \in \bQ$ there exists at most one exceptional vector bundle $E$ on $\cptComponent$ of slope $\slopeOfE = B$, and the set of exceptional slopes $\{\slopeOfE\}_E$ is completely determined in \cite{MR816365}.
Second, there exists a continuous periodic function $\dpFunction \colon \bR \to [\frac{1}{2}, 1]$ of period 1 such that the set
\begin{equation}
    \{(\slopeOfF, \discriminantOfF) \mid F \text{ is a slope-stable and non-exceptional sheaf on }\cptComponent\}
\end{equation}
coincides with the set $\{(\mu, \Delta) \in \bQ^2 \mid \Delta \geq \dpFunction(\mu)\}$.
See \cite[Section 2, Appendix A.1]{bayer-macri} for details.

Motivated by these observations, let $f \colon \{(a, b) \mid \Im a > 0\} \to \bR^2$ be the map defined by
\begin{equation}\label{eq:f}
    f(a, b) = \left(B(a, b), g(a, b)\right),
\end{equation}
and $G$ be the inverse image under $f$ of the points $(x, y) \in \bR^2$ that lie strictly below both of the following:
\begin{enumerate}
    \item $S_\infty = \{(\mu, \Delta) \mid \Delta \geq \dpFunction(\mu)\}$, and
    \item the discrete set $S_E = \{(\slopeOfE, \discriminantOfE) \mid E \text{ is an exceptional bundle on }\cptComponent\}$, which is disjoint from $S_\infty$.
\end{enumerate}
\begin{remark}\label{rem:slope-discriminant-curve}
    Every point $(\slopeOfE, \discriminantOfE) \in S_E$ lies below the curve $y = \dpFunction(x)$ which defines $S_\infty$.
\end{remark}

\begin{theorem}[{\cite[Section 4]{bayer-macri}}]\label{thm:total-space-geometric-chamber}
    The forgetful map $\cZ \colon \sigma_{a, b} \mapsto (a, b)$ gives an isomorphism from $\nTotalGeometricChamber$ to $G$.
    In particular, $\totalGeometricChamber$ is isomorphic to $G \times \bC$.
    Moreover, for any $\sigma \in \totalGeometricChamber$ the objects $\cO_x[n]$ for $x \in \cptComponent$ and $n \in \bZ$ are the only $\sigma$-stable objects whose class in $K(\simpleTotalDsupp)$ is equal to $\pm[\cO_x]$.
\end{theorem}

The structure of the boundary $\partial \totalGeometricChamber$ is given as follows.

For an exceptional bundle $E$ on $\cptComponent$, define a subset $\boundaryE$ of $\bC^2$ by
\begin{equation}
    \boundaryE =
    \left\{
      (a, b) \middle|
      \begin{aligned}
           & \Im a > 0,                                       \\
           & B(a, b) = \slopeOfE,                             \\
           & \discriminantOfE < g(a, b) < \dpFunction(B(a,b))
      \end{aligned}
    \right\}.
\end{equation}
In other words, $\boundaryE$ is the inverse image under $f$ \eqref{eq:f} of the vertical line segment connecting the point $(\slopeOfE, \discriminantOfE)$ and the curve $y = \dpFunction(x)$ (Remark \ref{rem:slope-discriminant-curve}).
A general point of $\partial G$ lies on some $\boundaryE$.

For $(a, b) \in \boundaryE$, we define a torsion pair $(\widetilde{\cT_E}, \widetilde{\cF_E})$ of $\Coh_{\cptComponent}(\totalSpace)$ as follows:
$\widetilde{\cT_E}$ is the extension-closed subcategory generated by $j_*E$, slope-semistable sheaves $F$ with $\slopeOfF > \slopeOfE$, and torsion sheaves, and
$\widetilde{\cF_E}$ is the extension-closed subcategory generated by slope-semistable sheaves $F$ with $\slopeOfF \leq \slopeOfE$ and $\Hom(j_*E, F) = 0$.
Then $\sigma_{a, b}^{+} = (Z_{a, b}, \langle \widetilde{\cF_E}[1], \widetilde{\cT_E} \rangle)$ defines a stability condition lying on the boundary $\partial \totalGeometricChamber$ \cite[Section 5]{bayer-macri}.
This construction gives rise to a wall
\begin{equation}
    \widetilde{W_E^+} = \overline{\{\sigma_{a, b}^{+} \cdot \lambda \mid (a, b) \in \boundaryE, \lambda \in \bC\}} \subset \partial \totalGeometricChamber.
\end{equation}
Together with $\widetilde{W_E^-} = T_E^{-1}(\widetilde{W_E^+})$, these walls describe the entire boundary $\partial \totalGeometricChamber$ as follows.
We denote the kernel of the evaluation map $j_*E^{\oplus r(E)} = \Hom(j_*E, \cO_x) \otimes j_*E \to \cO_x$ by $\widetilde{E^x}$.
\begin{theorem}[{\cite[Theorem 5.1]{bayer-macri}}]\label{thm:total-space-boundary}
    For every exceptional bundle $E$ on $\cptComponent$, there exist two codimension one walls $\widetilde{W_E^+}$ and $\widetilde{W_E^-}$ in $\partial \totalGeometricChamber$ such that
    \begin{enumerate}
        \item $\widetilde{W_E^+}$ consists of $\sigma = (Z, \cP)$ such that $j_*E$ and $\cO_x$ for $x \in \cptComponent$ are $\sigma$-semistable of the same phase $\phi$, with $j_*E$ being a subobject of $\cO_x$ in $\cP(\phi)$: at a general point of $\widetilde{W_E^+}$, the Jordan-H\"older filtration of $\cO_x$ is given by
              \begin{equation}
                  j_*E^{\oplus r(E)} \to \cO_x \to \widetilde{E^x}[1] \xrightarrow{+1},
              \end{equation}
        \item $\widetilde{W_E^-}$ consists of $\sigma = (Z, \cP)$ such that $j_*E[2]$ and $\cO_x$ for $x \in \cptComponent$ are $\sigma$-semistable of the same phase $\phi$, with $j_*E[2]$ being a quotient of $\cO_x$ in $\cP(\phi)$: at a general point of $\widetilde{W_E^-}$, the Jordan-H\"older filtration of $\cO_x$ is given by
              \begin{equation}
                  T_E^{-1}(\widetilde{E^x}[1]) \to \cO_x \to j_*E^{\oplus r(E)}[2] \xrightarrow{+1},
              \end{equation}
        \item $\widetilde{W_E^+} = \overline{\totalGeometricChamber} \cap T_E(\overline{\totalGeometricChamber})$ and $\widetilde{W_E^-} = \overline{\totalGeometricChamber} \cap T_E^{-1}(\overline{\totalGeometricChamber})$.
    \end{enumerate}
    Moreover, we have $\partial \totalGeometricChamber = \bigcup_E (\widetilde{W_E^+} \sqcup \widetilde{W_E^-})$ where $E$ runs through all exceptional bundles on $\cptComponent$.
\end{theorem}
Since $G$ is connected, there is a unique connected component $\totalStabDagger \subset \Stab(\simpleTotalDsupp)$ containing $\totalGeometricChamber = G \times \bC$.
The description of the boundary above implies the following.
\begin{theorem}[{\cite[Corollary 5.2]{bayer-macri}}]
    We have $\totalStabDagger = \bigcup_{\Phi} \Phi(\overline{\totalGeometricChamber})$ where $\Phi$ runs through the subgroup of $\Aut(\simpleTotalDsupp)$ generated by the spherical twists $T_E$ for all exceptional bundles $E$ on $\cptComponent$.
\end{theorem}

We also have the following.
\begin{theorem}[{\cite[Theorem 7.1]{bayer-macri}}]
    $\totalStabDagger$ is simply connected.
\end{theorem}

\subsection{Wall and chamber structure on $\Stab(\fiberDsupp)$}\label{subsec:wall-and-chamber-structure-on-fiber}
Following the previous subsection, we say that a stability condition $\sigma \in \Stab(\fiberDsupp)$ is \emph{geometric} if the skyscraper sheaves $\cO_x$ for all $x \in \cptComponent$ are $\sigma$-stable of the same phase, and denote the set of all geometric stability conditions on $\fiberDsupp$ by $\geometricChamber$.
By its definition and Lemma \ref{lem:lemma-on-inducing-stability-condition} (2), the map $\theta$ in Section \ref{subsec:comparison-isomorphism} induces an isomorphism between $\totalGeometricChamber$ and $\geometricChamber$, and hence between $\totalStabDagger$ and $\stabDagger$ where the latter is the connected component of $\Stab(\fiberDsupp)$ containing $\geometricChamber$.

\begin{corollary}\label{cor:fiber-stab-dagger-simply-connected}
    $\stabDagger$ is simply connected.
\end{corollary}
We also consider slope-stability for purely 2-dimensional sheaves in $\fiberCohSupp = \Coh_{\cptComponent}(\fiber)$.
Let $(\cT_B, \cF_B)$ be a torsion pair of $\fiberCohSupp$ as follows:
\begin{enumerate}
    \item $\cT_B$ is the extension-closed subcategory generated by slope-semistable sheaves $F$ with $\slopeOfF > B$ and torsion sheaves,
    \item $\cF_B$ is the extension-closed subcategory generated by slope-semistable sheaves $F$ with $\slopeOfF \leq B$.
\end{enumerate}
\begin{theorem}\label{thm:fiber-geometric-chamber}
    There is an isomorphism between $G \times \bC$ and $\geometricChamber$ given by
    \begin{equation}
        (a, b, \lambda) \mapsto (Z_{a, b}, \langle \cF_B[1], \cT_B \rangle) \cdot \lambda,
    \end{equation}
    where $B = -\frac{\Im b}{\Im a}$.
    Moreover, for any $\sigma \in \geometricChamber$ the objects $\cO_x[n]$ for $x \in \cptComponent$ and $n \in \bZ$ are the only $\sigma$-stable objects whose class in $K(\fiberDsupp)$ is equal to $\pm[\cO_x]$.
\end{theorem}
\begin{proof}
    By definition of $\theta$ and Theorem \ref{thm:total-space-geometric-chamber}, $\geometricChamber$ consists of the stability conditions of the form $\theta(\sigma_{a, b}) \cdot \lambda$ for some $(a, b) \in G$ and $\lambda \in \bC$.
    Let $\theta(\sigma_{a, b}) = (Z', \cP')$.
    We have an inclusion $\cP'((0, 1]) \supset \langle \cF_B[1], \cT_B \rangle$ by definition of slope.
    Since both sides are hearts of bounded t-structures, the inclusion is an equality.
    The rest follows from the second part of Theorem \ref{thm:total-space-geometric-chamber} and Lemma \ref{lem:lemma-on-inducing-stability-condition} (2).
\end{proof}

The next observation is crucial to describe the boundary $\partial \geometricChamber$.
\begin{lemma}\label{lem:half-spherical-twist-compatibility-on-stab}
    Let $E$ be an exceptional bundle on $\cptComponent$.
    Then the following diagram commutes:
    \begin{equation}
        \begin{tikzcd}
            \Stab(\simpleTotalDsupp) \ar[r, "\comparisonIso"] \ar[d, "T_{E}"'] & \Stab(\fiberDsupp) \ar[d, "H_{E}"] \\
            \Stab(\simpleTotalDsupp) \ar[r, "\comparisonIso"]                                       & \Stab(\fiberDsupp)
        \end{tikzcd}
    \end{equation}
\end{lemma}
\begin{proof}
    It follows from the relation $\closedImmersion_* \circ H_E \cong T_E \circ \closedImmersion_*$ (Proposition \ref{prop:half-spherical-twist-compatibility}) and the definition of $\comparisonIso$.
\end{proof}

For an exceptional bundle $E$ on $\cptComponent$, we denote the kernel of the evaluation map $\cptComponentImmersion_*E^{\oplus r(E)} = \Hom(\cptComponentImmersion_*E, \cO_x) \otimes \cptComponentImmersion_*E \to \cO_x$ by $E^x$.
Clearly, it satisfies $\closedImmersion_*E^x \cong \widetilde{E^x}$.
\begin{theorem}\label{thm:fiber-boundary}
    For every exceptional bundle $E$ on $\cptComponent$, there exist two codimension one walls $W_E^+$ and $W_E^-$ in $\partial \geometricChamber$ such that
    \begin{enumerate}
        \item $W_E^+$ consists of $\sigma = (Z, \cP)$ such that $\cptComponentImmersion_*E$ and $\cO_x$ for $x \in \cptComponent$ are $\sigma$-semistable of the same phase $\phi$, with $\cptComponentImmersion_*E$ being a subobject of $\cO_x$ in $\cP(\phi)$: at a general point of $W_E^+$, the Jordan-H\"older filtration of $\cO_x$ is given by
              \begin{equation}
                  \cptComponentImmersion_*E^{\oplus r(E)} \to \cO_x \to E^x[1] \xrightarrow{+1},
              \end{equation}
        \item $W_E^-$ consists of $\sigma = (Z, \cP)$ such that $\cptComponentImmersion_*E[2]$ and $\cO_x$ for $x \in \cptComponent$ are $\sigma$-semistable of the same phase $\phi$, with $\cptComponentImmersion_*E[2]$ being a quotient of $\cO_x$ in $\cP(\phi)$: at a general point of $W_E^-$, the Jordan-H\"older filtration of $\cO_x$ is given by
              \begin{equation}
                  H_E^{-1}(E^x[1]) \to \cO_x \to \cptComponentImmersion_*E^{\oplus r(E)}[2] \xrightarrow{+1},
              \end{equation}
        \item $W_E^+ = \overline{\geometricChamber} \cap H_E(\overline{\geometricChamber})$ and $W_E^- = \overline{\geometricChamber} \cap H_E^{-1}(\overline{\geometricChamber})$.
    \end{enumerate}
    Moreover, we have $\partial \geometricChamber = \bigcup_E (W_E^+ \sqcup W_E^-)$ where $E$ runs through all exceptional bundles on $\cptComponent$.
\end{theorem}
\begin{proof}
    It follows from Lemma \ref{lem:lemma-on-inducing-stability-condition}, Theorem \ref{thm:total-space-boundary}, and Lemma \ref{lem:half-spherical-twist-compatibility-on-stab}.
    Note that $j_*E = \closedImmersion_* \cptComponentImmersion_*E$.
\end{proof}

By an argument similar to that in Theorem \ref{thm:fiber-geometric-chamber}, a general element of the wall $W_E^+$ is of the form $(Z_{a, b}, \langle \cF_E[1], \cT_E \rangle) \cdot \lambda$ for some $(a, b) \in \boundaryE$ and $\lambda \in \bC$, where $(\cT_E, \cF_E)$ is the torsion pair of $\fiberCohSupp$ defined as follows:
\begin{enumerate}
    \item $\cT_E$ is the extension-closed subcategory generated by $\cptComponentImmersion_*E$, slope-semistable sheaves $F$ with $\slopeOfF > \slopeOfE$, and torsion sheaves,
    \item $\cF_E$ is the extension-closed subcategory generated by slope-semistable sheaves $F$ with $\slopeOfF \leq \slopeOfE$ and $\Hom(\cptComponentImmersion_*E, F) = 0$.
\end{enumerate}
To check that they form a torsion pair, we can apply the proof of \cite[Proposition 5.6]{bayer-macri},
with $\Hom_{\fiber}(\cptComponentImmersion_*E, \cptComponentImmersion_*E) = \bC$ and $\Ext^1_{\fiber}(\cptComponentImmersion_*E, \cptComponentImmersion_*E) = 0$ (by Lemma \ref{lem:end-ring-of-half-spherical-objects} below) in place of $\Hom_{\totalSpace}(j_*E, j_*E) = \bC$ and $\Ext^1_{\totalSpace}(j_*E, j_*E) = 0$.

\begin{lemma}\label{lem:end-ring-of-half-spherical-objects}
    Let $F \in \fiberDsupp$ be a half-spherical object.
    Then we have
    \begin{equation}
        \Ext^m_{\fiber}(F, F) = \begin{cases}
            \bC & \text{if $m = 0$ or $2$}, \\
            0   & \text{if $m<0$ or $m=1$}.
        \end{cases}
    \end{equation}
\end{lemma}
\begin{proof}
    It follows from the exact triangle
    \begin{equation}
        F[1] \to \closedImmersion^*\closedImmersion_*F \to F \xrightarrow{+1},
    \end{equation}
    the induced long exact sequence of $\Hom_{\fiber}(-, F)$, and the fact that $\closedImmersion_*F$ is a spherical object in $\totalDsupp$.
\end{proof}

Finally, one has a covering of $\stabDagger$ by translates of $\overline{\geometricChamber}$ as follows.
\begin{corollary}\label{cor:stab-dagger-is-covered-by-translates-of-geometric-chamber}
    We have $\stabDagger = \bigcup_{\Phi} \Phi(\overline{\geometricChamber})$ where $\Phi$ runs through the subgroup of $\Aut(\fiberDsupp)$ generated by the half-twists $H_E$ for all exceptional bundles $E$ on $\cptComponent$.
\end{corollary}
\begin{proof}
    It follows from Theorem \ref{thm:total-space-boundary} and Lemma \ref{lem:half-spherical-twist-compatibility-on-stab}.
    One can also prove it directly by the same argument as in \cite[Corollary 5.2]{bayer-macri} or \cite[Proposition 13.1]{MR2376815} using the description of the boundary $\partial \geometricChamber$.
\end{proof}

\section{Autoequivalence Group}\label{sec:autoequivalence-group}
We use the same notation for (half-)spherical twists as in Section \ref{sec:stability-conditions}.
For any autoequivalence $\Phi \in \Aut(\fiberD)$, we denote the induced autoequivalence of $\fiberDsupp$ (Proposition \ref{prop:auto-preserves-subcategories}) by the same letter $\Phi$.
Let $\Aut^\dagger(\fiberD)$ be the subgroup of $\Aut(\fiberD)$ consisting of autoequivalences $\Phi$ that preserve the connected component $\stabDagger \subset \Stab(\fiberDsupp)$.
The goal of this section is to prove the following theorem.
\begin{theorem}\label{thm:autoequivalence-group}
    \begin{equation}
        \Aut^\dagger(\fiberD) \cong \bZ \times \Gamma_1(3) \times \Aut(\fiber).
    \end{equation}
\end{theorem}

We first show that the subgroup of $\Aut^\dagger(\fiberD)$ generated by the shifts, $\Aut(\fiber)$, $\Pic(\fiber)$, and the half-twists $H_E$ for all exceptional bundles $E$ on $\cptComponent$ is isomorphic to the right-hand side of the above isomorphism.
Note that one has $H_E \in \Aut^\dagger(\fiberD)$ by Corollary \ref{cor:stab-dagger-is-covered-by-translates-of-geometric-chamber}.

\begin{lemma}\label{lemma:gamma1-3}
    The subgroup of $\Aut^\dagger(\fiberD)$ generated by $H_{\cO_{\cptComponent}}$ and $- \otimes\cO_{\fiber}(1)$ is isomorphic to $\Gamma_1(3)$.
\end{lemma}
\begin{proof}
    By \cite[Section 7.3.5]{MR2405589}, the functors $\alpha = T_{\cO_{\cptComponent}}$ and $\beta = -\otimes\canonicalBundleProjection^*\cO_{\canonicalBundleBaseSpace}(1)$ satisfy the relation $(\alpha\beta)^3 \cong \id$ at the level of Fourier--Mukai kernels and their convolution.
    Then there exists a group homomorphism
    \begin{equation}
        \Gamma_1(3) \to \Aut^\dagger(\fiberD)
    \end{equation}
    which maps $\alpha$ to $H_{\cO_{\cptComponent}}$ and $\beta$ to $-\otimes\cO_{\fiber}(1)$ (see Example \ref{ex:half-spherical-twist-conjugation} and the discussion above it).
    On the other hand, by Proposition \ref{prop:auto-preserves-subcategories}, one has a natural action of $\Gamma_1(3)$ on $K(\fiberDsupp)$ through $H_{\cO_{\cptComponent}}$ and $-\otimes\cO_{\fiber}(1)$.
    Now, recall that we have an isomorphism $i_* \colon K(\fiberDsupp) \xrightarrow{\sim} K(\totalDsupp)$ and the relations of functors $\closedImmersion_* \circ H_{\cO_{\cptComponent}} = T_{\cO_{\cptComponent}} \circ \closedImmersion_*$, $\closedImmersion_* \circ (-\otimes\cO_{\fiber}(1)) = (-\otimes\canonicalBundleProjection^*\cO_{\canonicalBundleBaseSpace}(1)) \circ \closedImmersion_*$.
    They show that the induced actions of $H_{\cO_{\cptComponent}}$ and $-\otimes\cO_{\fiber}(1)$ on $K(\fiberDsupp)$ are the same as those of $T_{\cO_{\cptComponent}}$ and $-\otimes\canonicalBundleProjection^*\cO_{\canonicalBundleBaseSpace}(1)$ on $K(\totalDsupp)$, i.e.~there is a commutative diagram
    \begin{equation}
        \begin{tikzcd}
            \Gamma_1(3) \ar[r, equal]& \langle T_{\cO_{\cptComponent}}, -\otimes\canonicalBundleProjection^*\cO_{\canonicalBundleBaseSpace}(1) \rangle \ar[r, twoheadrightarrow] \ar[d]& \langle H_{\cO_{\cptComponent}},-\otimes \cO_{\fiber}(1) \rangle \ar[d] \ar[r, phantom, "\subset"] &\Aut^\dagger(\cD)\\
            & \Aut(K(\totalDsupp)) \ar[r, equal] &  \Aut(K(\fiberDsupp)) &\\
        \end{tikzcd}
    \end{equation}
    of groups.
    In addition, the proof of \cite[Proposition 8.2]{bayer-macri} implies that the map $\Gamma_1(3) = \langle T_{\cO_{\cptComponent}}, -\otimes\canonicalBundleProjection^*\cO_{\canonicalBundleBaseSpace}(1) \rangle \to \Aut(K(\totalDsupp))$ in the above diagram is injective.
    Therefore, the group $\langle H_{\cO_{\cptComponent}}, -\otimes\cO_{\fiber}(1) \rangle \subset \Aut^\dagger(\fiberD)$ is isomorphic to $\Gamma_1(3)$.
\end{proof}

\begin{lemma}\label{lem:half-spherical-twist-in-gamma1-3}
    For any exceptional bundle $E$ on $\cptComponent$, the half-spherical twist $\halfTwistE$ is contained in the subgroup $\Gamma_1(3) = \langle H_{\cO_{\cptComponent}}, -\otimes\cO_{\fiber}(1) \rangle \subset \Aut^\dagger(\fiberD)$.
\end{lemma}
\begin{proof}
    This follows from the corresponding result \cite[Lemma 8.3]{bayer-macri} on $\totalSpace$, after restricting to the fiber.
    See the discussion in Example \ref{ex:half-spherical-twist-conjugation}.
\end{proof}

\begin{lemma}\label{lem:automorphism-commutes-with-half-spherical-twist}
    For any automorphism $f \in \Aut(\fiber)$, the induced autoequivalence $f_* \in \Aut^\dagger(\fiberD)$ commutes with the half-spherical twist $H_{\cO_{\cptComponent}}$.
\end{lemma}
\begin{proof}
    First, suppose that $f$ is contained in the subgroup $S_3 \ltimes \{(\lambda_1, \lambda_2, \lambda_3) \in (\bC^*)^3 \mid \lambda_1\lambda_2\lambda_3 = 1\} \subset \Aut(\fiber)$ (Proposition \ref{prop:automorphism-group-of-fiber}).
    By Proposition \ref{prop:automorphism-of-fiber-extends-to-total-space}, there exists an automorphism $\tilde{f}$ of $\totalDsupp$ over $\baseSpace$ such that $\tilde{f}\vert_\fiber = f$.
    For such $\tilde{f}$, we have the commutativity
    \begin{equation}
        T_{\cO_{\cptComponent}} \circ \tilde{f}_* = \tilde{f}_* \circ T_{\cO_{\cptComponent}}.
    \end{equation}
    By restricting the above relation to the fiber (Example \ref{ex:half-spherical-twist-conjugation}), one has
    \begin{equation}\label{eq:commutativity-of-automorphism-and-half-spherical-twist}
        H_{\cO_{\cptComponent}} \circ f_* = f_* \circ
        H_{\cO_{\cptComponent}}.
    \end{equation}

    Next, if $f$ is contained in the subgroup $\{(\lambda, \lambda, \lambda) \mid \lambda \in \bC^*\}$ of the $(\bC^*)^3$-part of $\Aut(\fiber) \cong S_3 \ltimes (\bC^*)^3$, then $f$ acts on $\cptComponent$ trivially (Remark \ref{rem:diagonal-automorphism-acts-trivially}).
    This implies that the relation \eqref{eq:commutativity-of-automorphism-and-half-spherical-twist} holds on the subcategory $\fiberDsupp \subset \fiberD$.
    On the other hand, since the kernel of $H_{\cO_{\cptComponent}}$ is supported on $\cptComponent \times \cptComponent$ and $f$ preserves the open subvariety $\fiber \setminus \cptComponent$, we have
    \begin{equation}\label{eq:automorphism-commutes-with-half-spherical-twist-on-points}
        H_{\cO_{\cptComponent}} \circ f_* (\cO_x) \cong f_* \circ H_{\cO_{\cptComponent}}(\cO_x)
    \end{equation}
    for all $x \in \fiber \setminus \cptComponent$.
    Then by Lemma \ref{lem:criterion-to-be-standard-functor}, there exists a line bundle $\cL$ on $\fiber$ such that
    \begin{equation}
        H_{\cO_{\cptComponent}} \circ f_* \circ H_{\cO_{\cptComponent}}^{-1} \circ (f_*)^{-1} \cong -\otimes \cL.
    \end{equation}
    Notice that $\cL$ is trivial if and only if its restriction to $\cptComponent \cong \bP^2$ is trivial (Proposition \ref{prop:picard-group-of-fiber}).
    Therefore, the relation \eqref{eq:commutativity-of-automorphism-and-half-spherical-twist} on $\fiberDsupp$ implies that $\cL \cong \cO_{\fiber}$.

    Finally, since $\Aut(\fiber)$ is generated by the above two types of automorphisms (Proposition \ref{prop:automorphism-group-of-fiber}), the relation \eqref{eq:commutativity-of-automorphism-and-half-spherical-twist} holds for all $f \in \Aut(\fiber)$.
\end{proof}
Next, we prove that all autoequivalences of $\fiberD$ are of Fourier--Mukai type.
Let us recall the notion of dg enhancement.
For a triangulated category $\cT$, a \emph{dg enhancement} of $\cT$ is a pair $(\cB, \varepsilon)$ where $\cB$ is a pretriangulated dg category and $\varepsilon \colon H^0(\cB) \to \cT$ is an exact equivalence.
We say that $\cT$ has a \emph{strongly unique dg enhancement} if it admits a dg enhancement and, for any two dg enhancements $(\cB, \varepsilon)$ and $(\cB', \varepsilon')$ of $\cT$, there exists a quasi-functor $F \colon \cB \to \cB'$ whose $H^0(F) \colon H^0(\cB) \to H^0(\cB')$ is an exact equivalence satisfying $\varepsilon' \circ H^0(F) \cong \varepsilon$.

\begin{lemma}
    The sequence of objects $\{\cO_{\fiber}(l)\}_{l \in \bZ}$ gives an ample sequence in $\Coh(\fiber)$ (in the sense of \cite[Definition 9.1]{MR2629991}), where $\cO_{\fiber}(1) = \canonicalBundleProjection^*\cO_{\canonicalBundleBaseSpace}(1)\vert_{\fiber}$ is an ample line bundle on $\fiber$.
\end{lemma}
\begin{proof}
    By \cite[Lemma 5.2]{ballard2009equivalencesderivedcategoriessheaves}, it is enough to check the following two conditions:
    \begin{enumerate}
        \item $\cO_{\fiber}(l) \neq \cO_{\fiber}$ for all $l \neq 0$.
        \item For sufficiently large $l$, we have $H^k(\fiber, \cO_{\fiber}(l)) = 0$ for all $k > 0$.
    \end{enumerate}
    (1) is clear.
    For (2), projection formula implies
    \begin{equation}
        H^k(\fiber, \cO_{\fiber}(l)) \cong H^k(\canonicalBundleBaseSpace, \canonicalBundleProjection_*\closedImmersion_*\cO_{\fiber} \otimes \cO_{\canonicalBundleBaseSpace}(l)).
    \end{equation}
    Note that the push-forward $\canonicalBundleProjection_*$ coincides with the non-derived one as $\canonicalBundleProjection$ is affine.
    Since there is an exact sequence
    \begin{equation}
        0 \to \cO_{\totalSpace} \to \cO_{\totalSpace} \to \closedImmersion_*\cO_{\fiber} \to 0,
    \end{equation}
    we only need to find an integer $l$ such that $H^k(\canonicalBundleBaseSpace, \canonicalBundleProjection_*\cO_{\totalSpace} \otimes \cO_{\canonicalBundleBaseSpace}(l)) = 0$ for all $k > 0$.
    By definition of $\totalSpace$, we have
    \begin{equation}\canonicalBundleProjection_*\cO_{\totalSpace} \cong \bigoplus_{m \geq 0} \Sym^m(\cO_{\canonicalBundleBaseSpace}(3)) \cong \bigoplus_{m \geq 0} \cO_{\canonicalBundleBaseSpace}(3m).
    \end{equation}
    Then $l \geq -2$ satisfies the condition.
\end{proof}
\begin{corollary}
    The category $\fiberD = D^b(\fiber)$ has a strongly unique dg-enhancement.
\end{corollary}
\begin{proof}
    Apply \cite[Proposition 7.13]{MR4853403} for the ample sequence $\{\cO_{\fiber}(l)\}_{l \in \bZ}$.
\end{proof}
\begin{proposition}[All autoequivalences of $\cD$ are Fourier-Mukai type]\label{prop:all-autoequivalences-are-FM-type}
    For any $\Phi \in \Aut(\fiberD)$, there exists a kernel $\cP \in D^b(\fiber \times \fiber)$ such that $\Phi \cong \Phi_\cP$.
\end{proposition}
\begin{proof}
    Let $\Perf_{dg}(\fiber) \subset D_{dg}^b(\fiber) \subset D_{dg}(\QCoh(\fiber))$ be dg-enhancements of $\Perf(\fiber) \subset D^b(\fiber) \subset D(\QCoh(\fiber))$.
    Since $D^b(\fiber)$ has a strongly unique dg-enhancement, there exists a quasi-functor $\Phi_{dg} \colon D_{dg}^b(\fiber) \to D_{dg}^b(\fiber)$ such that $H^0(\Phi_{dg}) \cong \Phi$.
    Then the functor $\Phi' \colon \Perf(\fiber) \to D(\QCoh(\fiber))$ induced by $\Phi$ has a lift to a quasi-functor
    \begin{equation}
        \Perf_{dg}(\fiber) \hookrightarrow D_{dg}^b(\fiber) \xrightarrow{\Phi_{dg}} D_{dg}^b(\fiber) \hookrightarrow D_{dg}(\QCoh(\fiber)).
    \end{equation}
    By \cite[Theorem 8.9]{MR2276263}, there exists a kernel $\cP \in D(\QCoh(\fiber \times \fiber))$ such that $\Phi \vert_{\Perf(\fiber)} =\Phi' \cong \Phi_\cP\vert_{\Perf(\fiber)}$ (see also \cite[Section 1]{MR3556457}).
    In particular, $\Phi$ and $\Phi_\cP$ are isomorphic when restricted to the full subcategory $\{\cO_{\fiber}(l) \mid l \in \bZ\}$ of $D^b(\fiber)$, and the Orlov's argument \cite[Proposition 2.16]{MR1465519} (or \cite[Proposition B.1]{MR2629991}) shows $\Phi \cong \Phi_\cP$ on the whole category $D^b(\fiber)$.
    The condition $\cP \in D^b(\fiber \times \fiber)$ follows, for example, from the proof of \cite[Corollary 9.13 (4)]{MR2629991}.
\end{proof}

Finally, we show that any autoequivalence $\Phi \in \Aut^\dagger(\fiberD)$ is contained in the right-hand side of the isomorphism in Theorem \ref{thm:autoequivalence-group}.
\begin{lemma}\label{lem:autoeq-preserves-Ox-class}
    For any autoequivalence $\Phi \in \Aut(\fiberD)$, the induced automorphism $\Phi \colon K(\fiberDsupp) \to K(\fiberDsupp)$ preserves the subgroup $\bZ[\cO_x]$.
\end{lemma}
\begin{proof}
    Let $\langle -, -\rangle$ be a bilinear form on $K(\fiberDsupp)$ given by
    \begin{equation}
        \langle [E], [F] \rangle = \chi_{\totalSpace}(\closedImmersion_*E, \closedImmersion_*F) = \sum_{k} (-1)^k \dim_\bC \Ext_{\totalSpace}^k(\closedImmersion_*E, \closedImmersion_*F).
    \end{equation}
    By Serre duality on $\totalSpace$, the form $\langle -, -\rangle$ is skew-symmetric, and the subgroup $\bZ[\cO_x]$ is characterized as its kernel.
    Here, we claim that $\langle E , F \rangle = 0$ if $F \in \fiberDsupp$ is a perfect complex.
    In that case, there exists a well-defined linear form
    $\chi_{\fiber}(-, F) \colon K(\fiberDsupp) \to \bZ$
    given by
    \begin{equation}
        \chi_{\fiber}(E, F) = \sum_{k} (-1)^k \dim_\bC \Ext_{\fiber}^k(E, F)
    \end{equation}
    since the sum is finite.
    We also have an exact triangle
    \begin{equation}
        F[1] \to \closedImmersion^*\closedImmersion_*F
        \to F \xrightarrow{+1}.
    \end{equation}
    Then by the adjunction $\closedImmersion^* \dashv \closedImmersion_*$ one has $\langle E, F \rangle = \chi_{\fiber}(\closedImmersion^*\closedImmersion_*F, E) = \chi_{\fiber}(E[1], F) + \chi_{\fiber}(E, F) = 0$.
    Now, pick a smooth point $x \in \fiber^{\sm} \cap \cptComponent$ so that $\cO_x$ and hence $\Phi(\cO_x)$ are perfect complexes (see proof of Lemma \ref{lem:auto-preserves-compact-support}).
    Then $\langle - , \Phi(\cO_x) \rangle = 0$ and thus $\Phi(\cO_x) \in \bZ[\cO_x]$.

\end{proof}

\begin{lemma}\label{lem:geometric-stability-condition-standard-equivalence}
    Let $\Phi \in \Aut^\dagger(\fiberD)$ be an autoequivalence such that there exist two geometric stability conditions $\sigma, \sigma' \in \geometricChamber$ with $\Phi(\sigma) = \sigma'$.
    Then $\Phi$ is of the form $f_*\circ (-\otimes \cL)[n]$ for some $f \in \Aut(\fiber)$, $\cL \in \Pic(\fiber)$ and $n\in \bZ$.
\end{lemma}
\begin{proof}
    Since the subgroup $\bZ[\cO_x] \subset K(\fiberDsupp)$ is preserved by $\Phi$ (Lemma \ref{lem:autoeq-preserves-Ox-class}), for every $x \in \cptComponent$ the object $\Phi(\cO_x)$ is a $\sigma'$-stable object of class $\pm[\cO_x]$.
    By Theorem \ref{thm:fiber-geometric-chamber}, it must be of the form $\cO_y[n]$ for some $y \in \cptComponent$ and $n \in \bZ$.
    Together with Corollary \ref{cor:outside-skyscraper-sheaf}, we can apply Lemma \ref{lem:criterion-to-be-standard-functor} so that $\Phi = f_*\circ (-\otimes\cL)[n]$ for some $f \in \Aut(\fiber)$, $\cL \in \Pic(\fiber)$, and $n \in \bZ$.
\end{proof}

\begin{proof}[Proof of Theorem \ref{thm:autoequivalence-group}]
    The subgroup of $\Aut^\dagger(\fiberD)$ generated by the shifts, $\Aut(\fiber)$, $\Pic(\fiber)$, and the half-twists $H_E$ for all exceptional bundles $E$ on $\cptComponent$ is isomorphic to $\bZ \times \Gamma_1(3) \times \Aut(\fiber)$ by Proposition \ref{prop:picard-group-of-fiber}, Lemma \ref{lemma:gamma1-3}, Lemma \ref{lem:half-spherical-twist-in-gamma1-3}, and Lemma \ref{lem:automorphism-commutes-with-half-spherical-twist}.
    Given $\Phi \in \Aut^\dagger(\fiberD)$, pick a geometric stability condition $\sigma \in \geometricChamber$.
    By Corollary \ref{cor:stab-dagger-is-covered-by-translates-of-geometric-chamber}, there exists another autoequivalence $\Psi \in \langle H_{\cO_{\cptComponent}}, -\otimes\cO_{\fiber}(1) \rangle = \Gamma_1(3)$ such that $(\Psi \circ \Phi)(\sigma) \eqcolon \sigma' \in \overline{\geometricChamber}$.
    Slightly moving $\sigma$ in $\geometricChamber$ if necessary, we may assume that $\sigma'$ is also a geometric stability condition.
    Then Lemma \ref{lem:geometric-stability-condition-standard-equivalence} shows $\Psi \circ \Phi \in \bZ \times \Gamma_1(3) \times \Aut(\fiber)$.
\end{proof}

\appendix
\section{Lemmas on Integral Functors}\label{sec:lemmas-on-integral-functors}
\begin{lemma}[cf.~{\cite[Corollary A.10]{MR3136492}}]\label{lem:quasi-finite-flat-with-constant-degree-is-finite}
    Let $f \colon X \to Y$ be a quasi-finite, separated, and flat morphism between noetherian schemes.
    Suppose the degree of any fiber of $f$ is constant.
    Then $f$ is finite.
\end{lemma}
Here we say that a morphism $f \colon X \to Y$ is \emph{quasi-finite} if it is of finite type and has finite fibers.
The degree of a fiber $X_y = f^{-1}(y)$ is $\deg(X_y) = \dim_{\kappa(y)} H^0(X_y, \cO_{X_y})$.
\begin{proof}[Proof of Lemma \ref{lem:quasi-finite-flat-with-constant-degree-is-finite}]
    We may assume the constant degree of $f$ is non-zero.
    By the valuative criterion, we may assume that $Y = \Spec R$ for a DVR $R$.
    By Zariski's main theorem \cite[\href{https://stacks.math.columbia.edu/tag/05K0}{Tag 05K0}]{stacks-project}, it factors as an open immersion $X \hookrightarrow \overline{X}$ followed by a finite morphism $\overline{f} \colon \overline{X} \to Y$ (and hence $\overline{X}$ is affine).
    We may assume that $X$ is dense in $\overline{X}$.
    We can also assume that $\overline{X}$ is flat over $Y$ as follows:
    Let $R \to \overline{A} \to H^0(\cO_X)$ be the ring homomorphisms corresponding to $Y \leftarrow \overline{X} \hookleftarrow X$, where $\overline{X} = \Spec \overline{A}$.
    The flatness of $X$ over $R$ implies that $H^0(\cO_X)$ is torsion-free as $R$-module, and thus we can replace $\overline{A}$ with its torsion-free quotient so that $\overline{X}$ is flat over $Y$ \cite[\href{https://stacks.math.columbia.edu/tag/0539}{Tag 0539}]{stacks-project}.

    Let $y$ and $\eta$ be the closed point and the generic point of $Y$, respectively.
    Since $\overline{A}$ is a finite free $R$-module \cite[\href{https://stacks.math.columbia.edu/tag/02KB}{Tag 02KB}]{stacks-project}, we have $\deg(\overline{X}_y) = \deg(\overline{X}_\eta)$.
    We also have $\deg(X_y) = \deg(X_\eta)$ by the assumption on $f$, and $\deg(X_\eta) = \deg(\overline{X}_\eta)$ since $X \subset \overline{X}$ is dense.
    These equalities show $\deg(X_y) = \deg(\overline{X}_y)$, which implies $X_y = \overline{X}_y$.
    Thus, $X = \overline{X}$ is finite over $Y$.
\end{proof}
\begin{lemma}[A stronger version of {\cite[Lemma 3.11]{arai2024halfspherical}}]\label{lem:criterion-to-be-standard-functor}
    Let $k$ be an algebraically closed field with characteristic zero.
    Let $X$ and $Y$ be connected quasi-projective schemes over $k$, with $X$ reduced.
    Suppose we are given an integral functor $\Phi = \Phi_{\cP}$ with $\cP \in D^b(X \times_k Y)$ that satisfies the following condition:
    \begin{quote}
        For any closed point $x \in X$, there exists a closed point $y \in Y$ and an integer $n_x$ such that $\Phi(\cO_x) \cong \cO_y[n_x]$.
    \end{quote}
    Then there exists a morphism $f \colon X \to Y$, a line bundle $\cL \in \Pic X$, and an integer $n$ such that $\Phi \cong f_*(- \otimes \cL)[n]$.
\end{lemma}
\begin{proof}
    To apply \cite[Lemma 3.11]{arai2024halfspherical}, we must show that $\cP$ has a proper support over $X$.
    Let $\Sigma = \Supp(\cP) \subset X \times_k Y$ be the support of $\cP$ and $p \colon \Sigma \to X$ be the projection.
    The assumption on $\Phi$ implies that $p$ is flat and induces a bijection on closed points, and hence it is quasi-finite with constant degree $1$.
    Then Lemma \ref{lem:quasi-finite-flat-with-constant-degree-is-finite} implies that $p$ is finite, and therefore proper.
\end{proof}

\printbibliography

\end{document}